\documentclass[11pt, a4paper, reqno]{amsart}

\usepackage{amsmath,amsthm,amssymb}
\usepackage[nobysame, alphabetic]{amsrefs}

\usepackage[margin=30mm]{geometry}

\usepackage{stackengine}

\usepackage{enumerate}
\usepackage{tikz}
\usetikzlibrary{cd, arrows}
\usepackage[arrow,matrix,curve,cmtip,ps]{xy}
\usepackage{pb-diagram}
\usepackage{pb-xy}
\usepackage{subcaption}
\usepackage{xcolor}
\usepackage{esvect}
\usepackage{thmtools}
\usepackage{thm-restate}

\definecolor{darkblue}{rgb}{0,0,0.6}

\usepackage[breaklinks, pdftex, ocgcolorlinks,colorlinks=true, citecolor=darkblue, filecolor=darkblue, linkcolor=darkblue, urlcolor=darkblue]{hyperref}
\usepackage[capitalize,noabbrev]{cleveref}

\newtheorem{proposition}{Proposition}[section]
\newtheorem{theorem}[proposition]{Theorem}
\newtheorem{corollary}[proposition]{Corollary}
\newtheorem{lemma}[proposition]{Lemma}

\theoremstyle{definition}
\newtheorem{definition}[proposition]{Definition}
\newtheorem{question}[proposition]{Question}
\newtheorem{example}[proposition]{Example}

\newtheorem{construction}[proposition]{Construction}

\theoremstyle{remark}
\newtheorem{remark}[proposition]{Remark}

\newtheorem*{remark*}{Remark}

\numberwithin{equation}{section}

\newcommand{\sm}{\setminus}

\newcommand{\Q}{\mathbb{Q}}

\newcommand{\R}{\mathbb{R}}
\newcommand{\Z}{\mathbb{Z}}

\newcommand{\Id}{\operatorname{Id}}
\newcommand{\pt}{\operatorname{pt}}

\newcommand{\id}{\operatorname{Id}}

\newcommand{\wh}{\widehat}

\let\int\relax
 \newcommand{\int}{\mathring}

\newcommand{\ol}{\overline}
\newcommand{\wt}{\widetilde}

\newcommand{\CP}{\mathbb{C}P}
\newcommand{\RP}{\mathbb{R}P}

\newcommand{\smfrac}[2]{\mbox{\footnotesize$\displaystyle\frac{#1}{#2}$}} % small medium frac
\def\HMf{\widehat{HM}}

\RequirePackage{rotating}                   % Case (2)
    \def\HMt{%
       \setbox0=\hbox{$\widehat{\mathit{HM}}$}
       \setbox1=\hbox{$\mathit{HM}$}
       \dimen0=1.1\ht0
       \advance\dimen0 by 1.17\ht1
       \smash{\mskip2mu\raise\dimen0\rlap{%
          \begin{turn}{180}
              {$\widehat{\phantom{\mathit{HM}}}$}
           \end{turn}} \mskip-2mu
                \mathit{HM}
    }{\vphantom{\widehat{\mathit{HM}}}}{}}

\DeclareMathOperator{\Hom}{Hom}

\DeclareMathOperator{\Diff}{Diff}
\DeclareMathOperator{\Emb}{Emb}
\DeclareMathOperator{\Homeo}{Homeo}
\DeclareMathOperator{\FSW}{FSW}

\DeclareMathOperator{\Wh}{Wh}
\DeclareMathOperator{\spin}{spin}

\newcommand{\spineQ}{Q^{\rm sp}}

\usepackage{xcolor}

\begin{document}

\title[Corks for diffeomorphisms]{Corks for exotic diffeomorphisms}

 \author[V.\ Krushkal]{Vyacheslav Krushkal}
 \address{Department of Mathematics, University of Virginia, Charlottesville, VA 22903, USA}
 \email{krushkal@virginia.edu}

\author[A.\ Mukherjee]{Anubhav Mukherjee}
\address{Department of Mathematics, Princeton University, Princeton, NJ 08540, USA}
\email{anubhavmaths@princeton.edu}

 \author[M.\ Powell]{Mark Powell}
 \address{School of Mathematics and Statistics, University of Glasgow, United Kingdom}
 \email{mark.powell@glasgow.ac.uk}

 \author[T.\ Warren]{\\Terrin Warren}
 \address{Department of Mathematics, California Polytechnic State University, San Luis Obispo, CA, 93407, USA }  \email{terrinwarren@gmail.com}

\def\subjclassname{\textup{2020} Mathematics Subject Classification}
\expandafter\let\csname subjclassname@1991\endcsname=\subjclassname
%\expandafter\let\csname subjclassname@2000\endcsname=\subjclassname
%\keywords{4-manifolds, diffeomorphisms, corks}

\begin{abstract}
We prove a localization theorem for exotic diffeomorphisms, showing that every diffeomorphism of a compact simply-connected 4-manifold that is isotopic to the identity after stabilizing with one copy of $S^2 \times S^2$,  is smoothly isotopic to a diffeomorphism that is supported on a contractible submanifold. 
For those that require more than one copy of $S^2 \times S^2$, we prove that the diffeomorphism can be isotoped to one that is supported in a submanifold homotopy equivalent to a wedge of 2-spheres, with null-homotopic inclusion map. 
We investigate the implications of these results by applying them to known exotic diffeomorphisms. 
\end{abstract}

\maketitle

\section{Introduction}\label{sec:intro}

This article concerns exotic diffeomorphisms of simply-connected 4-mani-folds, and our goal is to investigate to what extent they can be localized.

Let $X$ be a compact, simply-connected, smooth 4-manifold. 
We say that a diffeomorphism $f \colon X \to X$ is \emph{exotic} if it is topologically but not smoothly isotopic to the identity~$\Id_X$. If $X$ has a nonempty boundary then we will assume that diffeomorphisms and isotopies are \emph{boundary fixing}, that is they restrict to $\Id_{\partial X}$, for all time in the case of isotopies.  We say that a diffeomorphism $f \colon X \rightarrow X$ is \emph{supported on a submanifold $C \subseteq X$} if $f$ restricts to the identity on the complement of its interior $X \sm \mathring{C}$. We say that $f$ is \emph{$n$-stably isotopic to $\Id$} if $f\# \id \colon X\#^n S^2\times S^2 \to X\#^n S^2\times S^2$ is smoothly isotopic to the identity. 
A diffeomorphism $f \colon X \to X$ is topologically isotopic to $\Id$ if and only if $f$ is $n$-stably isotopic to $\Id$ for some~$n$; see \cref{sec: stable diffeomorphisms} for details and citations. 

The Cork Theorem~\cites{CFHS,Matveyev}, states that any exotic pair of compact, simply-connected, smooth 4-manifolds $X$ and $X'$ are related by a \emph{cork twist}, i.e.\ there is a compact, contractible, smooth codimension zero submanifold $C\subseteq X$, the eponymous cork, with an involution $\tau \colon \partial C \to \partial C$, such that $X \sm \mathring{C} \cup_\tau C \cong X'$. The first cork was discovered by Akbulut~\cite{Akbulut-cork}, and they have been studied extensively, e.g.~\cites{Akbulut-Matveyev,Akbulut-Yasui-08,Akbulut-Yasui-09,Akbulut-Yasui-13,Gompf-infinite-order-corks}. The Cork Theorem was extended to any finite collection of smooth structures by Melvin-Schwartz~\cite{Melvin-Schwartz}. 

Our main result is an analogue of the cork theorem for diffeomorphisms. 

\begin{restatable}[Diffeomorphism cork theorem]{theorem}{oneeyethm}
\label{thm:one-eye-thm}
 Let $X$ be a compact, simply-connected, smooth 4-manifold, and let $f \colon X \rightarrow X$ be a boundary fixing diffeomorphism such that $f$ is 1-stably isotopic to $\Id$.  
    Then there exists a compact, contractible submanifold $C \subseteq X$, and a boundary fixing isotopy of $f$ to a  diffeomorphism $f' \colon X \to X$ that is supported on $C$. 
\end{restatable}

A more detailed statement, \cref{thm:one-eye-thm-full-version}, strengthens the result in two ways. Namely, we show that any finite collection of diffeomorphisms that are $1$-stably isotopic to $\Id$ can be localized to a compact, contractible submanifold $C$. Moreover,  $C$ can be chosen to be a 4-manifold that admits a handle decomposition into 0-, 1-, and 2-handles.

A \emph{diff-cork} is a pair $(C,g)$ consisting of a smooth, compact, contractible $4$-manifold $C$ together with a diffeomorphism $g \colon C \to C$ such that $g|_{\partial C} = \Id_{\partial C}$. In the terminology of \cref{thm:one-eye-thm}, $g = f'|_C$. Note that for any diff-cork, the diffeomorphism $g\colon C \to C$ is topologically isotopic to the identity by \cites{Perron,Quinn:isotopy}.

\begin{remark}
The contractible 4-manifold $C$ of a diff-cork is not a cork in the sense of \cites{Kirby_corks,Akbulut-Matveyev,Akbulut-Yasui-08}, as $C$ does not come with the data of an involution of the boundary, although this would arguably be unnatural to expect in this context.
\end{remark}

The proof of the classical cork theorem involves analyzing the structure of an $h$-cobordism from $X$ to $X'$, decomposing it into a contractible sub-$h$-cobordism and a product cobordism. Our proof of \cref{thm:one-eye-thm} is somewhat analogous, where a pseudo-isotopy plays the role of the $h$-cobordism. 
Recall that a \emph{pseudo-isotopy} of $X$ is a diffeomorphism $F \colon X \times I \to X \times I$ such that $F$ restricts to the identity on $X \times \{0\} \cup \partial X \times I$. We say that $f:= F|_{X \times \{1\}}$ is \emph{pseudo-isotopic} to $\Id$.  

A diffeomorphism of a compact, simply-connected 4-manifold that is stably isotopic to $\Id$ is pseudo-isotopic to $\id$; see \cref{thm:TFAE-isotopy-notations} for a detailed discussion and citations.
A result by Gabai~\cite{gabai20223spheres} implies that if the diffeomorphism is 1-stably  isotopic to the identity then the pseudo-isotopy  can be assumed to have one eye. This  means that it admits an associated Cerf graphic with one eye, corresponding to a birth and subsequent death of a single pair of critical points of indices 2 and 3 (see \cref{section:preliminaries} for further details). So to prove \cref{thm:one-eye-thm}  it suffices to study one-eyed pseudo-isotopies.  
We analyze the structure of a one-eyed pseudo-isotopy and show that it can be decomposed into a pseudo-isotopy supported on $C \times I \subseteq X \times I$, where $C$ is a contractible submanifold of $X$. 

\begin{remark}
    Our method can be contrasted with that of Gay~\cite{gay2021diffeomorphisms} (in the case of~$S^4$) and Krannich-Kupers~\cite{krannich-kupers} (for arbitrary simply-connected 4-manifolds). They characterized exotic diffeomorphisms via embedding spaces. Their proofs also employed pseudo-isotopy theory, but the outcomes are rather different.
\end{remark}

When $f$ must be stabilized by more than one copy of $S^2 \times S^2$ 
in order to smoothly trivialize it, we do not know whether there is a cork theorem.  We can instead localize the diffeomorphism to a 4-manifold homotopy equivalent to a wedge of 2-spheres, whose inclusion in $X$ is null-homotopic.

\begin{restatable}{theorem}{genbarbellthm}\label{thm:gen-barbell-cork-thm}
    Let $X$ be a smooth, compact, simply-connected 4-manifold, and let $f \colon X\to X$ be a diffeomorphism  that is $n$-stably isotopic to identity. Then there exists $k\leq {n(n-1)}$ and a compact 4-manifold $\mathcal{B}$ and a smooth embedding $\iota \colon \mathcal{B} \to X$, such that $\iota \colon \mathcal{B} \to X$ is null-homotopic, $\vee^k S^2 \simeq \mathcal{B}$, and such that $f$ is smoothly isotopic to a diffeomorphism supported on~$\iota(\mathcal{B})$.
\end{restatable}

It would be interesting to know whether this result is optimal. To this end, consider a $4$-dimensional Dehn twist $\delta$ along the separating $3$-sphere in $K_3\#K_3$. 
This diffeomorphism is topologically isotopic to the identity \cites{Kreck-isotopy-classes,Perron,Quinn:isotopy,GGHKP},  not smoothly isotopic to the identity \cite{Kronheimer-Mrowka-K3}, and  moreover not 1-stably isotopic to the identity \cite{Lin-dehn-twist-stabn}. This leads to the following question. 

\begin{question}
    Can $\delta$ be isotoped to a diffeomorphism of $K_3 \# K_3$ supported on a contractible 4-manifold?  
\end{question}

More examples of non 1-stably isotopic exotic diffeomorphisms were constructed in \cites{Konno-Mallick-Taniguchi}, and the same question applies to these diffeomorphisms as well.

\subsection*{Applications of the Diffeomorphism Cork Theorem} 
The first examples of exotic diffeomorphisms of simply-connected 4-manifolds are due to Ruberman~\cites{ruberman-isotopy,Ruberman-polynomial-invariant}. These examples were shown to be $1$-stably isotopic to the identity by Auckly-Kim-Melvin-Ruberman~\cite{AKMR-stable-isotopy}*{Theorem~C}. 
We check in \cref{Example:BK,Example:Auckly} that this also holds for the examples of Baraglia-Konno from~\cite{Baraglia-Konno} and those of Auckly from~\cite{auckly:stable-surfaces}.  

Then, as a consequence of \cref{thm:one-eye-thm}, each of these examples  admits a diff-cork $(C,g)$, and $g$ is an exotic diffeomorphism of~$C$. %

We note that the existence of exotic diffeomorphisms on contractible 4-manifolds was first shown by Konno-Mallick-Taniguchi~\cite{Konno-Mallick-Taniguchi}. 
However our construction allows us to take finite collections of Ruberman's diffeomorphisms, and obtain the following result.

\begin{restatable}{theorem}{Zmdiffcork}\label{thm:Zm-in-a-diff-cork} For each $m \geq 1$ there exists a contractible, compact, smooth 4-manifold~$\mathcal{C}_m$ and a collection $\{g_1,\dots,g_m\}$ of boundary-fixing diffeomorphisms of~$\mathcal{C}_m$ that generate a subgroup of $\pi_0 \Diff_{\partial}(\mathcal{C}_m)$ that abelianizes to~$\Z^m$. 
\end{restatable}

Here the subscript $\partial$ indicates that all maps fix the boundary pointwise. 
\cref{thm:Zm-in-a-diff-cork} produces subgroups of mapping class groups of contractible 4-manifolds that determine arbitrarily large but finite rank subgroups of the abelianization, by localizing families of diffeomorphisms of closed 4-manifolds with nonzero second Betti number. 
Konno-Mallick~\cite{Konno-Mallick-localizing} proved that  localizing  cannot produce infinite rank subgroups, so \cref{thm:Zm-in-a-diff-cork} is in this sense optimal. 

Relatedly, we mention an application of corks for diffeomorphisms, due to Shivkumar~\cite{Shivkumar-compactly-supp-diffeos}, that was inspired by our article. He showed that diff-corks can be assumed to have simply-connected complements. Combined with a proof strategy first used in \cite{KLMME}*{p.~64}, he applied this to construct exotic copies of $\R^4$ with compactly-supported smooth mapping class groups of arbitrarily large rank.  

Now we come to our final application of \cref{thm:one-eye-thm}.
Galatius and Randal-Williams \cite{galatius2023alexander}*{Theorem~B} proved that for dimension $n$ at least 6, and $C$ a contractible, compact, smooth $n$-manifold, the extension map $\Diff_\partial(D^n)\to \Diff_\partial(C)$ is a weak equivalence, for any embedding $D^n \hookrightarrow \mathring{C}$ of the $n$-dimensional disc into the interior of $C$.  The analogous statement was proven for $n=5$ by Krannich-Kupers~\cite{Krannich-Kupers-operadic-foundations}*{Theorem~6.18}. 
We show that this does not hold in dimension $4$.

\begin{restatable}{theorem}{inclusionnotweak}\label{thm:inclusion-not-he}
     There exists a smooth, compact, contractible $4$-manifold $C$ and a smooth embedding $D^{4} \subseteq C$ such that the extension map $\Diff_{\partial}(D^4) \hookrightarrow \Diff_{\partial} (C)$ is not surjective on path components, so is not a weak equivalence. 
\end{restatable}

Galatius--Randal-Williams also show that $\Homeo_\partial(D^n) \to \Homeo_\partial(C)$ is a weak equivalence for $n \geq 6$ ~\cite{galatius2023alexander}*{Theorem~A}, and again this was extended to $n=5$ in~\cite{Krannich-Kupers-operadic-foundations}. It is unknown whether this holds in dimension 4. 
However we do have an isomorphism $\pi_0\Homeo_{\partial}(D^4) \xrightarrow{\cong} \pi_0\Homeo_{\partial} (C)$ for any contractible compact 4-manifold~$C$ by \cite{Quinn:isotopy}*{Proposition~2.2} together with \cite{Perron}, \cite{Quinn:isotopy}*{Theorem~1.4}.

Following a suggestion of David Gabai, we prove the following result. In contrast with \cref{thm:inclusion-not-he}, it shows that every exotic diffeomorphism of $\natural^n S^2 \times D^2$ is induced from $D^4$. 
We will recall the definition of barbell diffeomorphisms in \cref{section:generalised-barbell}.

\begin{restatable}{theorem}{barbellgenerates} \label{thm:barbell-generate}
For the 4-manifold $X_n := \natural^n S^2 \times D^2$, $n\geq 1$, there is an exact sequence
\[\pi_0\Diff_{\partial}(D^4) \to \pi_0 \Diff_{\partial}(X_n) \to \pi_0 \Homeo_{\partial}(X_n) \to 0.\]
Moreover, $\pi_0 \Homeo_{\partial}(X_n)$ is generated by standard barbell diffeomorphisms $\phi_{i,j}$ for $1 \leq i <j \leq n$. 
\end{restatable}

\begin{remark}
    \cref{thm:barbell-generate} implies that if there exists an exotic diffeomorphism of $\natural^n S^2\times D^2$ for some $n$, this would immediately produce an exotic diffeomorphism of the 4-ball. 
\end{remark}

\subsubsection*{Organization}

In \cref{section:preliminaries} we recall the necessary background on pseudo-isotopies and their connections to 1-parameter families of Morse functions on $X \times I$, as well as to stable diffeomorphisms.  
In \cref{section:structure-of-PI} we analyse the structure of pseudo-isotopies, introducing the Quinn core, and we prove several key lemmas. 
Then in \cref{section:1-eyed-cork-thm} we  prove \cref{thm:one-eye-thm}. 
\cref{section:sum-square-move-and-pi2-element} recalls Quinn's sum square move and defines a collection of elements of $\pi_2(X)$ determined by a nested eye Cerf family.  
In \cref{section:many-eyed-theorem}, making use of these methods, we prove \cref{thm:gen-barbell-cork-thm}. 
\cref{section:examples-1-stable} gives a unified treatment of several examples of exotic diffeomorphisms of simply-connected 4-manifolds from the literature that are 1-stably isotopic to the identity, and proves \cref{thm:Zm-in-a-diff-cork}.  
\cref{section:FSW-applications}  shows that these examples give rise to diffeomorphisms of contractible 4-manifolds with nontrivial family Seiberg-Witten invariants, and proves \cref{thm:inclusion-not-he}. Finally \cref{section:generalised-barbell} proves \cref{thm:barbell-generate}.

\subsubsection*{Acknowledgements}

We are grateful to David Auckly, Michael Freedman, David Gabai, David Gay, Daniel Hartman, Ailsa Keating, Hokuto Konno, Alexander Kupers, Roberto Ladu, Francesco Lin, Abhishek Mallick, and Danny Ruberman for many helpful conversations and suggestions. We are especially grateful to Hokuto Konno for discussing his related work and for detailed comments on a previous draft. We would also like to thank the anonymous referee for helpful suggestions which improved the paper.

\cref{thm:one-eye-thm} appeared independently in the PhD thesis of the fourth-named author, who would like to thank David Gay and David Auckly for their encouragement and helpful suggestions.  
This collaboration formed after the two teams learned that they had proven similar results.

VK was supported in part by NSF grants DMS-2105467 and DMS-2405044. 
AM was supported in part by NSF grant DMS-2405270. 
MP was supported in part by EPSRC grants EP/T028335/2 and EP/V04821X/2. TW was supported in part by NSF grant DMS-2005554 and by Simons Foundation grant MP-TSM-00002714. VK and MP would like to thank the International Centre for Mathematical Sciences in Edinburgh for hospitality and support during some of the work on this paper.

%added this 
%\subsection*{Open Access}
%For the purpose of open access, the authors have applied a CC-BY Creative Commons attribution license to this author-accepted manuscript.  

\section{Pseudo-isotopies and Cerf families of generalized Morse functions}\label{section:preliminaries}

Let $X$ be a smooth, compact 4-manifold. A \emph{pseudo-isotopy} on $X$ is a diffeomorphism $F \colon X \times I \xrightarrow{\cong} X \times I$ that restricts to the identity on $X \times \{0\} \cup \partial X \times I$.
If such a diffeomorphism preserves level sets $X\times \{s\}$ for all $s\in I$, then it is a smooth isotopy. 

\subsection{From pseudo-isotopies to Cerf families}
From a pseudo-isotopy $F$, we obtain a family of functions and gradient-like vector fields denoted by $\{(q_{t}, v_{t})\}$ and constructed as follows. Let $q_0 \colon X\times I \to I$ be the projection $q_0(x,s)= s$, and let $v_{0}$ be the unit vector field $\partial / \partial s$
on $X \times I$. Define \[(q_{1}, v_{1}) := (q_0 \circ F^{-1}, DF(v_0)).\]  
Both $q_0$ and $q_1$ are Morse functions without critical points. 

There is a generic 1-parameter family of generalized Morse functions $q_t \colon X \times I \to I$, along with an associated family of gradient-like vector fields $v_{t}$, interpolating between $(q_0,v_0)$ and $(q_1,v_1)$ ~\cite{Cerf}*{Section 4}. Here, a \emph{generalized Morse function} is permitted, unlike a Morse function, to have isolated degenerate critical points, however, they are singularities of codimension at most 1, corresponding to births and deaths of critical points. We call such a family $\{(q_t,v_t)\}_{t \in [0,1]}$ a \emph{Cerf family}.

Since $q_1$ is a Morse function with no critical points, we can integrate $v_1$ to obtain a diffeomorphism of $X \times I$. In fact, this recovers the pseudo-isotopy $F$, as we explain in the next lemma. 

\begin{lemma}\label{lemma:how-to-recover}
The diffeomorphism $X\times I \to X \times I$ obtained by flowing upwards from $X\times \{ 0\}$ along the vector field $v_1$ is precisely the diffeomorphism $F \colon X \times I \rightarrow X \times I$. 
\end{lemma}

\begin{proof}
Recall that $q_1 := q_0 \circ F^{-1} \colon X \times I \to I$ and $v_1 := DF(v_0)$. Fix $p \in X \times \{0\}$, and let $\alpha_p \colon I \to X \times I$ be the integral curve of $v_1$. That is, $\alpha_p$ is the unique solution to the ODE \[\smfrac{\mathrm{d}}{\mathrm{d}s} \alpha_p(s) = v_1(\alpha_p(s)) \in T_{\alpha_p(s)} (X \times I)\] 
 with initial condition $\alpha_p(0)=p$.  Observe that $s \mapsto F(p,s)$ satisfies the same differential equation, i.e.\ \[\smfrac{\partial}{\partial s} F(p,s) = DF_{(p,s)}(v_0) =:  v_1(F(p,s)) \in T_{F(p,s)}(X \times I).\] Then uniqueness of solutions to ODEs implies that $\alpha_p(s) = F(p,s)$ for all $s \in I$.   
\end{proof}

\begin{remark}\label{remark:get-Id-integrating}
Similarly, the diffeomorphism $X\times I \to X \times I$ obtained by flowing upwards from $q_0^{-1}(0)$ along $v_0$ is the identity $\id_{X \times I}$.  
\end{remark}

\begin{remark}\label{remark:no-crit-pts-gives-isotopy}
If $q_{t}$ has no critical points for all $t$, then by \cref{lemma:how-to-recover} and \cref{remark:get-Id-integrating}, integrating $v_t$ yields a  family of diffeomorphisms $F_t \colon X \times I \rightarrow X \times I$ interpolating between $\id_{X \times I}$ and $F$.
By the smooth dependence of the solutions of ODEs on initial conditions, this is a smooth family, i.e.\ a smooth isotopy.  
The restriction of $F_t$ to the top slice, $X \times \{1\}$, gives an isotopy $f_{t} \colon X \rightarrow X$, with $f_{0} = \id$ and  $f_{1} = f$.  
However, typically $q_{t}$ will have critical points for some $t$. 
\end{remark}

Let $t \in (0,1)$ and let $p \in X \times I$ be a critical point of $q_t$. Let $Y_{p,t} \subseteq X \times \{0\}$ be the set of points $x \in X \times \{0\}$ such that the trajectory of $v_t$ starting at $x$ limits to $p$. The following lemma will be used in the discussion of the Quinn core in \cref{section:structure-of-PI}.

\begin{lemma}\label{lemma:pseudo-isotopy-support}
Let $F \colon X \times I \rightarrow X \times I$ be a pseudo-isotopy and let $(q_t, v_t)$ be a Cerf family for $F$. Let $Y \subseteq \mathring{X} \times \{0\}$ be a compact codimension zero submanifold such that
\[\bigcup_{t\in (0,1),\, p \textup{  crit. pt. of }q_t} Y_{p,t} \subseteq \mathring{Y}.\]
Then $F$ is isotopic to a pseudo-isotopy that is supported in $Y \times I$. 
\end{lemma}

\begin{proof}
Define $Z := X \sm \mathring{Y}$. 
    If $X$ has nonempty boundary, then  $\partial X \times I \subseteq Z \times I$ since $Y$ lies in the interior of $X$. Since, for all $t$, $q_t$ has no critical points in $Z$, the flow along $v_t$ gives a family of embeddings $\varphi_t \colon Z \times I \rightarrow X \times I$ with $\varphi_t (Z \times \{0\}) = Z \times \{0\}$ and $\varphi_t (Z \times \{1\}) \subseteq X \times \{1\}$. We can fix $q_t$ and $v_t$ to be such that %\footnote{MP: rather than just we can assume that, perhaps we could fix $q_t$ and $v_t$ to be such that this holds automatically?} 
    $\varphi_{t}$ restricts to the  identity on $Z \times \{0\} \cup \partial X \times I$.  Note that $\varphi_0 \colon Z \times I \to X \times I$ is the standard embedding, by definition of $q_0$ and $v_0$.   
    Using the isotopy extension theorem, extend $\varphi_{t}$ to a family of diffeomorphisms $G_{t} \colon X \times I \rightarrow X \times I$ with $G_0 = \id$ and $G_t \circ \varphi_0 = \varphi_t$ for all $t$. Using the relative version of the isotopy extension theorem, we can further arrange that $G_t$ restricts to the identity on $X \times \{0\}$ for all $t \in [0,1]$. Then \[F' := G_1^{-1} \circ F  \colon X \times I \rightarrow X \times I\] is a pseudo-isotopy, with an isotopy $G_t^{-1} \circ F$ from $F= G_0^{-1} \circ F$ to $F' = G_1^{-1} \circ F$.  
    Moreover,~$F'$ is supported in $\mathring{Y} \times I$. To see this, note that by \cref{lemma:how-to-recover} we have that $F|_{Z\times I} = \varphi_1$. Since $G_1 \circ \varphi_0 = \varphi_1$ we have that $G_1^{-1}\circ \varphi_1 = \varphi_0$. Hence 
    \[F'|_{Z \times I} = G_1^{-1} \circ F|_{Z \times I} = G_1^{-1} \circ \varphi_1 = \varphi_0 = \Id_{Z \times I}.\]
    It follows that $F'$ is supported in the complement of $Z \times I$, namely $\mathring{Y} \times I$. 
\end{proof}

\subsection{Nested eye graphics}\label{subsection:nested-eye-families}
Hatcher and Wagoner~\cite{HW}  introduced the secondary Whitehead group $\Wh_2(\pi)$ of a group $\pi$, and defined an obstruction $\Sigma(F) \in \Wh_2(\pi_1(X))$ of a  pseudo-isotopy of~$X$. 

When $\Sigma (F)$ vanishes, Hatcher and Wagoner showed that one can deform the 1-parameter family $(q_t,v_t)$ until its Cerf diagram is a \textit{nested eye} diagram with critical points of index 2 and 3 only.
Here, for each $t$, all critical points are assumed to have distinct critical values, and apart from birth and death times critical points of index 2 have critical values below those of critical points of index 3. Moreover, a nested eye diagram has the following features. 
\begin{enumerate}[(i)]
    \item There are $n$ birth points, of canceling index 2 and 3 pairs of critical points, for some $n\geq 0$.
    \item There are no rearrangements, and then all $n$ pairs cancel against one another.
    \item  There are independent birth and death points, and no handle slides. 
\end{enumerate}

\begin{figure}[ht]
\centering
\includegraphics[height=3.8cm]{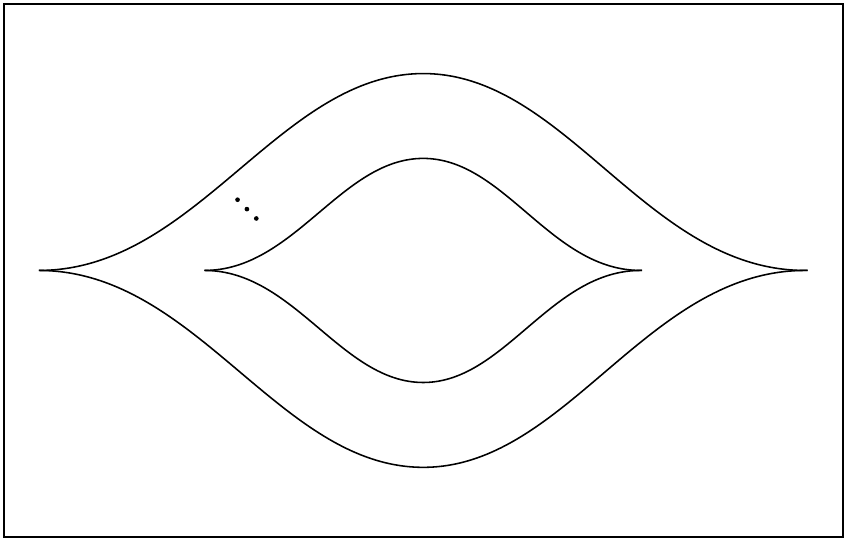}
\caption{A Cerf graphic for a family in nested eye position. The horizontal direction is the $t$-axis and the vertical direction is the $[0,1]$ direction, recording the critical values of the critical points in the Cerf family. } 
\label{figure:Cerf}
\end{figure}

For $\pi$ the trivial group, $\Wh_2(\{e\})=0$, and so the Hatcher-Wagoner obstruction $\Sigma(F)$ necessarily vanishes when $X$ is a simply-connected 4-manifold.  
Consequently, there is always a deformation of $(q_t,v_t)$ to a nested eye family with indices~$2$ and~$3$. Throughout the rest of this article all Cerf families will be assumed to be in nested eye position.  

\subsection{Stable diffeomorphisms}\label{sec: stable diffeomorphisms} 
We recall some  results on stable diffeomorphisms of $4$-manifolds used in follow-up sections.

The main results of this paper, in particular \cref{thm:one-eye-thm,thm:gen-barbell-cork-thm}, are stated for diffeomorphisms that are stably isotopic to the identity. 
The main applications of these results are to exotic diffeomorphisms, namely those that are topologically but not smoothly isotopic to the identity. Indeed, for a simply-connected $4$-manifold $X$, a diffeomorphism $f \colon X \to X$ is topologically isotopic to $\Id$ if and only if $f$ is $n$-stably isotopic to $\Id$ for some~$n$. We recall how to deduce this from the literature in the following theorem, and we elucidate the relationship with smooth and topological pseudo-isotopy. 

\begin{theorem}\label{thm:TFAE-isotopy-notations}
    Let $f \colon X \to X$ be a diffeomorphism of a simply-connected, compact, smooth 4-manifold, and if $\partial X \neq \emptyset$ then assume that $f|_{\partial X} = \Id_{\partial X}$. The following are equivalent:
    \begin{enumerate}[(i)]
      \item\label{item:TFAE-i} $f$ is topologically isotopic rel.\ boundary to $\Id$; 
        \item\label{item:TFAE-ii} $f$ is smoothly pseudo-isotopic rel.\ boundary to $\Id$;
        \item\label{item:TFAE-iii} $f$ is smoothly stably isotopic rel.\ boundary to $\Id$; 
        \item\label{item:TFAE-iv} $f$ is topologically pseudo-isotopic rel.\ boundary to $\Id$.
    \end{enumerate}
\end{theorem}

\begin{proof}
Suppose \eqref{item:TFAE-i}, that $f$ is topologically isotopic to $\id$. Then $f$ acts trivially on the integral homology of $X$. If $X$ is closed,  it was shown by Kreck~\cite{Kreck-isotopy-classes}, and later by Quinn~\cite{Quinn:isotopy} (with a correction by Cochran-Habegger~\cite{Cochran-Habegger}), that $f$ is smoothly pseudo-isotopic to the identity. 
For $X$ with nonempty, connected boundary, one also observes that $f$ has trivial Poincar\'{e} variation, and it was shown by Saeki~\cite{Saeki} that $f$ is smoothly pseudo-isotopic to the identity. This was generalized to arbitrary boundary by Orson-Powell~\cite{Orson-Powell} by also noting that $f$ acts trivially on relative spin structures. This proves that \eqref{item:TFAE-i} implies \eqref{item:TFAE-ii}.

Then, assuming \eqref{item:TFAE-ii} it follows from Quinn~\cite{Quinn:isotopy} (with a correction in \cite{GGHKP}), and independently Gabai~\cite{gabai20223spheres}*{Theorem~2.5}, that $f$ is smoothly stably isotopic to the identity (rel.\ boundary). This proves that \eqref{item:TFAE-ii} implies \eqref{item:TFAE-iii}.

Now suppose that \eqref{item:TFAE-iii} holds, namely that $f$ is smoothly stably isotopic to the identity. Then, the invariants of $f$ from the first paragraph vanish, and by \cite{Kreck-isotopy-classes}, \cite{Quinn:isotopy}, and \cite{Orson-Powell}, $f$ is smoothly pseudo-isotopic to the identity. So \eqref{item:TFAE-ii} and \eqref{item:TFAE-iii} are equivalent. 

It is immediate that \eqref{item:TFAE-ii} implies \eqref{item:TFAE-iv}. 
So it remains to see that \eqref{item:TFAE-iv} implies \eqref{item:TFAE-i}. 
For this, Perron~\cite{Perron} and Quinn~\cite{Quinn:isotopy} (with a different correction to the latter in \cite{GGHKP}) showed that if $f$ is topologically pseudo-isotopic to the identity then it is topologically isotopic to the identity.
\end{proof}

The following theorem of Gabai gives a quantitative version of \cref{thm:TFAE-isotopy-notations}~\eqref{item:TFAE-ii}$\Longleftrightarrow$\eqref{item:TFAE-iii}. 

\begin{theorem}[{\cite{gabai20223spheres}*{Theorem~2.5~and~Corollary~2.10}}]\label{lemma:n-stable-isotopy-iff-n-eyed-PI}
Let $f \colon X \rightarrow X$ be a diffeomorphism with $f|_{\partial X} = \Id$. Then 
\[f\# \id \colon X\#^n (S^2\times S^2) \rightarrow X\#^n (S^2\times S^2)\] 
is smoothly isotopic to the identity rel.\ boundary if and only if $f$ is pseudo-isotopic to the identity via an $n$-eyed pseudo-isotopy. 
\end{theorem}

\begin{remark}
Note that in order to define $f\#\id$ one has to make a choice of an isotopy of~$f$ to a diffeomorphism that fixes a $4$-ball. It was shown in \cite{AKMR-stable-isotopy}*{Theorem~5.3} that all choices give rise to isotopic stabilized diffeomorphisms, i.e.\ that $f\#\id$ is well-defined up to isotopy.
\end{remark}

\section{The structure of a pseudo-isotopy}\label{section:structure-of-PI}

In this section we discuss the structure of the middle-middle level and its relation to the rest of the pseudo-isotopy. 

 \subsection{The Quinn core} \label{sec: the Quinn core}
Let $F \colon X \times I \rightarrow X \times I$ be a pseudo-isotopy of a simply-connected $4$-manifold with a nested eye family $(q_t,v_t)$.  
For each $t$ such that $q_t$ is a Morse function (which holds for all but the finitely many values of $t$ where births and deaths occur), the data $(q_t,v_t)$ determines a handle decomposition of $X \times I$ obtained by attaching $5$-dimensional $2$- and $3$-handles to $X \times [0,\varepsilon]$.
We describe the properties of this family of handle structures. 

Assume that there are $n$ nested eyes, and that all births happen before 
$t = 1/4$ and all deaths happen after $t = 3/4$.  Using Cerf's uniqueness of birth and death lemmas~\cite{Cerf}*{Chap.~III}, after a deformation of the family we assume that the Cerf data is given by the standard model births and deaths for a small interval of time around the births and deaths, $t \in (1/4-\delta,1/4)$ and $t \in (3/4,3/4 + \delta)$ respectively. The standard model births or deaths take place in respective $5$-balls in $X\times I$, and after a deformation we assume that $(q_t, v_t)$ is constant away from these balls for $t \in (1/4-\delta,1/4)$ and $t \in (3/4,3/4 + \delta)$.

For each $t \in [1/4,3/4]$, let $M_t := q_t^{-1}(1/2)$ denote the middle level set between the $2$- and $3$-handles. 
We call $M_{1/2}$ the \emph{middle-middle level}.  

Let $A^{t} := \{A^{t}_1, \ldots, A^{t}_{n}\}$ denote the ascending spheres of the 2-handles  and let $B^{t} := \{B^{t}_1, \ldots, B^{t}_{n}\}$ denote the descending spheres of      the $3$-handles, all in $M_t$.  For each $t \in [1/4, 3/4]$, the middle level $M_t$ is    diffeomorphic to \[M:=X \#^{n} S^{2} \times S^{2}.\] 

\begin{construction}\label{construction:identification-of-the-middle-level}
Throughout this paper we will use the following identification of $M_t$ with $M$ for $1/4\leq t \leq 3/4$.  A closely related argument was given in \cite{gay2021diffeomorphisms}*{Proof~of~Theorem~9}, for the case $X=S^4$.

Consider the framed  attaching circles $\alpha^t:=\sqcup_{i=1}^n \alpha^t_i \colon \sqcup^n S^1 \times D^3 \to X\times \{ \varepsilon\}$ of the $2$-handles at time $t$. 
Use  the product   structure on $X\times [0, \varepsilon ]$ to view the circles $\alpha^t$ as  embeddings 
$\sqcup_i\alpha^t_i \colon \sqcup^n S^1 \times D^3 \to  X\times \{ 0\}$.  

It is convenient to introduce the notation \[\gamma := \alpha^{1/4} \colon \sqcup^n S^1 \times D^3 \to X \] for the embeddings of thickened circles at time $t=1/4$.  We sometimes abuse notation and conflate $\gamma$ and $\alpha^t$  with their images, which are $n$ disjointly embedded framed circles $\gamma = \sqcup_i \gamma_i\subseteq X$.  
Let $X_{\gamma}$ denote the result of surgering $X$ along the framed circles $\gamma$
\begin{equation} \label{Xgamma}
X_{\gamma} := X \sm \gamma(\sqcup^n S^1 \times \mathring{D}^3) \cup_{\gamma|^{}_{\partial}} (\sqcup^n D^2 \times S^2). \end{equation}
Since there are embedded discs in $X$, with boundaries the $\gamma_i$, across which the framing extends, it follows that the result of the surgery is diffeomorphic to $M = X \#^n S^2 \times S^2$.  
We use a fixed choice of such discs to fix once and for all an identification \[M = X_\gamma.\]

Let $\varphi_t$ be an isotopy of $X\times \{ 0 \}$, $t\in[1/4, 3/4]$, that for each $t$ takes $\gamma$ to $\alpha^t$. We obtain this by applying the isotopy extension theorem to the isotopy of embeddings $\alpha^t$.  For each $t$, the diffeomorphism $\varphi_t$ induces a diffeomorphism 
\begin{equation}\label{eqn:diffeo-on-complemement-of-attaching-circles}
    X \sm \gamma(\sqcup^n S^1 \times \mathring{D}^3) \xrightarrow{\cong} X \sm \alpha^t(\sqcup^n S^1 \times \mathring{D}^3).
\end{equation}

Next, the flow of $v_t$ induces a diffeomorphism $X\sm \alpha^t(\sqcup^n S^1\times \mathring{D}^3) \xrightarrow{\cong}  M_t\sm \mathcal{N}(A)$, where~$\mathcal{N}(A)$ is an open tubular neighborhood of $A$ diffeomorphic to $\sqcup^n S^2\times \mathring{D}^2$.
We obtain a diffeomorphism because we removed neighborhoods of the ascending and descending manifolds of the index 2 critical points, so the flow encounters no critical points.  
Combining these two diffeomorphisms, we obtain:
\[{X\sm \gamma(\sqcup^n S^1\times \mathring{D}^3)} \xrightarrow{\cong}  {X\sm \alpha^t(\sqcup^n S^1\times \mathring{D}^3)} \xrightarrow{\cong}  {M_t\sm {\mathcal N}(A)}. \]
Attaching $\sqcup^n D^2 \times S^2$ to the domain yields $X_{\gamma}$, as in \eqref{Xgamma}.
We extend the diffeomorphism from \eqref{eqn:diffeo-on-complemement-of-attaching-circles} over $\sqcup^n D^2 \times S^2$.  
The manifold $M_t$ is obtained from $M$ by surgery along $\alpha^t$, and hence using the flow of $v_t$ and the standard fact that passing a critical point gives rise to a surgery, we obtain an identification
\[M =X_\gamma \xrightarrow{\cong} M_t\]
for each $t \in [1/4,3/4]$.
This concludes the construction of an identification of $M$ with $M_t$.
\end{construction}

We continue the discussion of the geometric data in the middle level.
At $t = 1/4$ and $t= 3/4$, the spheres intersect transversely with $|A^{t}_i \pitchfork B^{t}_{j}| = \delta_{ij}$. As $t$ varies from $1/4$ to $3/4$, we see a regular homotopy of $A^t \cup B^t$ that restricts to an isotopy of the $A$-spheres and to an isotopy of the $B$-spheres.
During this regular homotopy, new intersection points are introduced by finger moves,  and removed by  Whitney moves. We can assume, after a deformation, that all finger moves are performed at time $t=3/8$ and all the Whitney moves are performed at time $t=5/8$. 
The finger and Whitney moves are guided by two collections of disjointly embedded discs in the middle-middle level $M_{1/2}$,  pairing excess intersections among $A^{1/2}$ and $B^{1/2}$, the finger discs $V$ and the Whitney discs $W$. Since a finger move with time reversed is a Whitney move, both collections of discs can be used as the data for a collection of Whitney moves to cancel excess intersections between $A^{1/2}$ and $B^{1/2}$. 
The moves corresponding to $V$ are performed with time reversed at $t=3/8$, while the moves corresponding to $W$ are performed at $t=5/8$. 
Each of the collections $V$ and $W$ consists of framed Whitney discs. 
For ease of notation, at $t=1/2$ we will suppress the mention of $t$ and denote the collections of spheres in $M_{1/2}$ by $A = \{A_1, \ldots, A_n\}$ and $B = \{B_1, \ldots, B_n\}$. 

\begin{definition}\label{defn:Quinn-core}
Define the {\em Quinn core} to be a regular neighborhood of $A \cup B \cup V \cup W$ in the middle-middle level,
\begin{equation} \label{eq:Q}
Q := \mathcal{N}(A \cup B \cup V \cup W ) \subseteq M_{1/2}=X \#^n S^{2} \times S^{2}.
\end{equation}
\end{definition}

Considering the trajectories of $v_{1/2}$ in $X\times I$ at $t=1/2$ that intersect the Quinn core, together with the trajectories starting from the index 3 critical points or ending at index 2 critical points, determines a sub-$h$-cobordism $P \subseteq X \times I$. This $h$-cobordism is obtained from $Q\times [-\varepsilon, \varepsilon] \subseteq M_{1/2} \times [1/2-\varepsilon,1/2 + \varepsilon]$ by attaching two collections of $3$-handles: to $Q\times \{ \varepsilon \}$ along $B$, and to $Q\times \{ -\varepsilon \}$ along $A$.
For $i=0,1$ we define \[Q_i := P \cap (X \times \{i\}).\]
The following statement is implicit in  
\cite{Quinn:isotopy}, and will be used to establish conventions and describe the framework used to prove \cref{thm:one-eye-thm}.

\begin{lemma}[Quinn Core Lemma]\label{lem:Quinn Core Lemma}
    Let $F \colon X \times I \rightarrow X \times I$ be a smooth pseudo-isotopy of a simply-connected $4$-manifold.  
    Then there is a smooth isotopy $F \simeq F'$ such that $F' = \id$ on $(X \sm \mathring{Q_0}) \times I$. 
\end{lemma}

\begin{proof}
   Recall from \cref{construction:identification-of-the-middle-level} that $M_t$ is identified with $M$ for $1/4\leq t\leq 3/4$, and as discussed in \cref{sec: the Quinn core}, for some small $\delta$ the times $t \in (1/4-\delta,1/4)$ and $t \in (3/4,3/4 + \delta)$ correspond to standard births and deaths.

   Next we focus on the main part of the proof, stating that after a deformation of the pseudo-isotopy, the isotopy of $A$ and $B$ spheres is confined to $Q$. For $3/8  <  t  <  5/8$ the union $A\cup B$ moves by an isotopy, but the topology of $A\cup B$ changes at times $t=3/8$ and $t=5/8$ when finger and Whitney moves take place, respectively. For $1/4\leq t\leq 3/8$ it is convenient to consider the union $Q_V^t:=\mathcal{N}(A^t\cup B^t \cup \upsilon^t)$, a regular neighborhood of the union of $A^t$, $B^t$, and arcs $\upsilon^t$ guiding the finger moves that will occur at $t=3/8$. 
   Reversing time, a Whitney move becomes a finger move, so for $5/8 \leq t \leq 3/4$ there are analogous arcs which we denote by~$\omega^t$, shown in \cref{figure:Whitney_move_spine}, and we consider the regular neighbourhood of the corresponding union $Q_W^t := \mathcal{N}(A^t\cup B^t\cup\omega^t)$. 
   \begin{figure}[ht]
\centering
\includegraphics[height=2cm]{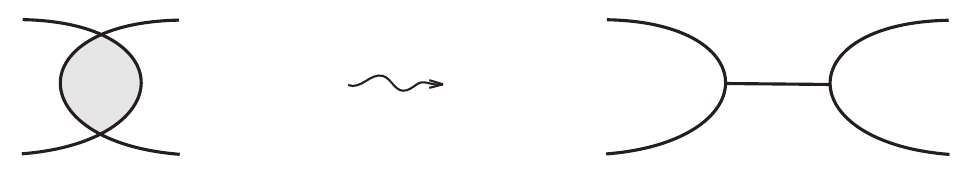}
{\small
\put(-322,0){$A^t$}
\put(-252,0){$B^t$}
\put(-288,25){$W^t$}
\put(-67,33){$\omega^t$}
\put(-133,0){$A^t$}
\put(-2,0){$B^t$}
}
\caption{Left: a schematic illustration of $A^t\cup B^t\cup W^t$ for $t$ just before the Whitney move time $5/8$. Right: $A^t\cup B^t\cup \omega^t$ for $t$ right after $5/8$.}
\label{figure:Whitney_move_spine}
\end{figure}
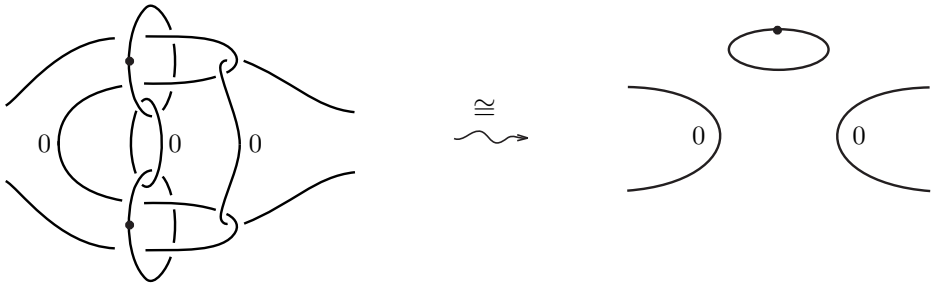
\begin{figure}[ht]
\centering
\includegraphics[height=4cm]{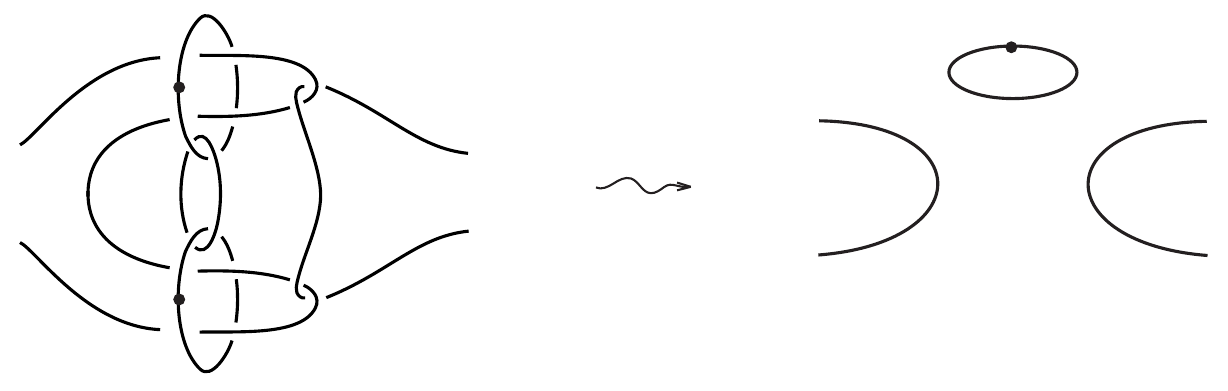}
{\small
\put(-340,54){$0$}
\put(-261,54){$0$}
\put(-291,54){$0$}
\put(-95,57){$0$}
\put(-35,57){$0$}
}
\put(-177,67){$\cong$}
\caption{The regular neighborhoods $\mathcal{N}(A^t\cup B^t\cup W^t)$, $t=5/8-\varepsilon$, and ${\mathcal N}(A^t\cup B^t\cup \omega^t)$, $t=5/8+\varepsilon$, are diffeomorphic. In terms of Kirby diagrams, a diffeomorphism is implemented by sliding the $0$-framed $2$-handle corresponding to either $A^t$ or $B^t$ (the left-most and right-most handles respectively) twice over the central $2$-handle corresponding to $W^t$, and then canceling a $1$-, $2$-handle pair.}
\label{figure:Whitney_move}
\end{figure}

For $t \in [1/2,5/8]$, we let $Q_W^t := \mathcal{N}(A^t \cup B^t \cup W^t)$, using a family of regular neighborhoods that vary smoothly with respect to~$t$. 
Similarly for $t \in [3/8,1/2]$ we define $Q_V^t := \mathcal{N}(A^t \cup B^t \cup V^t)$. 
\cref{figure:Whitney_move} shows Kirby diagrams for $Q_W^{5/8 - \varepsilon}$ and $Q_W^{5/8 + \varepsilon}$. The diagram for $Q_W^{5/8 - \varepsilon}$ on the left can be deduced, for example, from \cite{Matveyev}*{Figure~4}. In more detail, the figure in \cite{Matveyev} has two $1$-handles which may be called the accessory $1$-handle and the Whitney $1$-handle. Our diagram is obtained
by a handle slide of the accessory $1$-handle over the Whitney $1$-handle. The Whitney circle then links both of them, and it serves as the attaching circle of the zero-framed $2$-handle corresponding to the Whitney disc.

Observe that the two \emph{a priori} different versions of $Q_W^{t}$, for $t$ near the Whitney move time $5/8$, are in fact isotopic codimension zero submanifolds of $M_{5/8}$. To show this, \cref{figure:Whitney_move} indicates two handle slides and a cancellation to pass from a Kirby diagram for $Q_W^{5/8 - \varepsilon}$ to $Q_W^{5/8 + \varepsilon}$. 
These are local handle moves that can be implemented ambiently, yielding an isotopy of the codimension zero submanifold. 
By realizing this isotopy sufficiently close to $t=5/8$, we obtain a consistent family of regular neighborhoods, that we use the same notation $Q_W^t$ to describe. That is, for each $t \in [1/2,3/4]$, we have that $Q_W^t$ a codimension zero submanifold of $M_t \cong M$.  We similarly obtain such a family $Q_V^t$ in $M_t$ for $t \in [1/4,1/2]$.

We give the rest of the argument for $Q_W^t$ and $t \geq 1/2$. The same argument applies, with time reversed, for $Q_V^t$ and $t \leq 1/2$. 

Using the identification $M_t=M$ in \cref{construction:identification-of-the-middle-level}, for $t \in [1/2,3/4]$, we consider~$Q_W^t\subseteq M$. This determines an isotopy of ${\mathcal N}(A\cup B\cup W)$ in $M$, and using isotopy extension we obtain a corresponding isotopy $\varphi_t$ of $M$.  The effect of the inverse $\varphi^{-1}_t$ of this isotopy is that $Q^t_W$ becomes constant in $M=M_t$ for all $t\in[1/2, 3/4]$.

Next we extend this to an isotopy $\Phi_t$ of $X \times I$, $t \in [1/2,3/4]$, supported in a neighborhood  of $\cup_{t \in [1/2,3/4]} M_t$. Such a neighborhood is illustrated as the shaded region labeled (i) in \cref{figure:Isotopy}. 
\begin{figure}[ht]
\centering
\includegraphics[height=4.5cm]{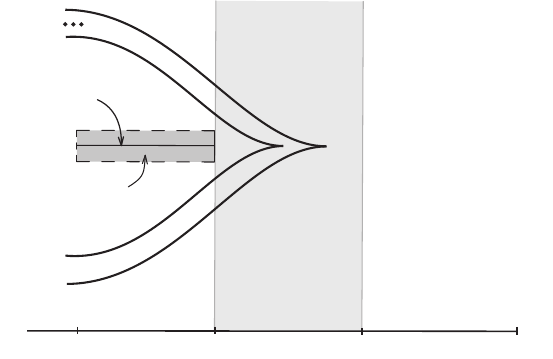}
{\scriptsize
\put(-100,40){(ii)}
\put(-171,97){$\varphi_t^{-1}$}
\put(-160,54){(i)}
\put(-50,71){(iii)}
\put(-176,-5){1/2}
\put(-122,-5){3/4}
\put(-80,-5){$3/4+\delta$}
\put(-7,-5){1}
}
\caption{The isotopy is given by $\Phi_t^{-1}$ in the preimage under $q_t$ of the shaded region labeled (i), its reverse in region (iii), and it is constant (as a function of $t$) in region (ii). The isotopy is the identity on the left, top, and bottom boundary arcs of the shaded rectangle (i), i.e.\ the dashed boundary arcs.  
The picture is symmetric for $t\leq 1/2$. }
\label{figure:Isotopy}
\end{figure}
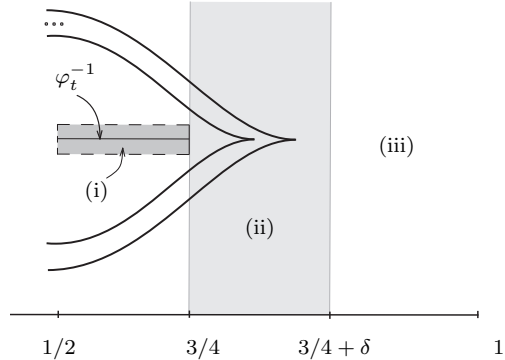

For $t \in [3/4,3/4+\delta]$, we assume that the family $(q_t,v_t)$ consists of $n$ elementary paths of death~\cite{Cerf} (where as usual $n$ is the number of eyes), each of which is supported in an arbitrarily small neighborhood of the corresponding death point.

Apply the inverse $\Phi^{-1}_t$ of this isotopy on $X\times I$ for $t$ in $[1/2,3/4]$, apply the constant isotopy for $t\in [3/4,3/4+\delta]$, and then undo $\Phi_t^{-1}$, i.e.\ apply $\Phi_{r(t)}^{-1}$ for $t \in [3/4+\delta,1]$, where $r \colon [3/4+\delta,1] \to [1/2,3/4]$ is the unique decreasing linear bijection. We are not concerned with the effect of $\Phi_t$ for $t \in [3/4+\delta,1]$ since all deaths occur before then. Note that the overall result can be deformed to not having applied any isotopy.  
By differentiating, we obtain a deformation of the family of gradient-like vector fields $\{v_t\}_{t \geq 1/2}$. 

The outcome of this operation is that $Q_W^t$ becomes constant in $M=M_t$ for all $t \in [1/2,3/4]$, and because we made the corresponding modification of $v_t$ we still have that $A^t \cup B^t \subseteq Q^t_W$ for all $t \in [1/2,3/4]$. 
Then since $Q_W^t = Q_W^{1/2} \subseteq Q$, we have that 
\[A^t \cup B^t \subseteq Q \subseteq M = M_t\] 
for all $t \in [1/2,3/4]$. 
As stated above, by reversing time we apply the analogous deformation of $v_t$ for $t \leq 1/2$, to arrange that $Q_V^t$ is constant and $A^t \cup B^t \subseteq Q_V^t$ for all $t \in [1/4,1/2]$. 
This completes the argument for our claim that, after a deformation, we can assume that the isotopy of $A$ and $B$ spheres in $M$ is confined to the Quinn core $Q$.

These deformations ensure that for all $t \in [1/4,3/4]$ the trajectories starting or ending at the critical points of $q_t$ are such that their intersection with the middle level lies in $Q$. Hence for every critical point of $q_t$ the downward flow along $v_t$ lands in $Q_0$.  We claim that this holds for all $t \in [0,1]$, or equivalently for all $t \in [1/4-\delta,3/4 + \delta]$, since there are no critical points outside this range. 

By the above considerations, we know that for $t = 3/4$ the downwards trajectories of all critical points of $X \times I$ land in $Q_0 \subseteq X \times \{0\}$.  By choosing the neighborhood for each elementary path of death to be sufficiently small, we guarantee that these trajectories land in $Q_0$ for all $t\in [3/4,3/4+\delta]$ as well. 
By symmetry the same conclusion holds for all $t\in [1/4-\delta,1/4]$. 

We then apply  \cref{lemma:pseudo-isotopy-support} to deduce that the pseudo-isotopy $F$ is isotopic to a pseudo-isotopy $F' \colon X \times I \to X \times I$ that is supported on $Q_{0} \times [0,1]$. This concludes the proof of \cref{lem:Quinn Core Lemma}.
\end{proof}

\subsection{The homotopy type of the Quinn core and the arc condition}\label{subsection: homotopy type of core}
Consider two spheres $A_i$ and $B_j$. Choose the orientations of $A_i$ and $B_i$ to be such that the algebraic intersection number $\lambda(A_i,B_j) = \delta_{ij}$. We write $V_{ij}$ (respectively $W_{ij}$) for the collection of the finger (respectively Whitney) discs pairing intersections of $A_i$ with $B_j$. 
Assuming genericity, the intersection $(\partial V_{ij} \cup  \partial W_{ij}) \cap A_i$  is the image of a generic immersion $\sqcup^{C_{ij}} S^1 \looparrowright A_i$, for some integers $C_{ij} \in \mathbb{N}_0$, and if $i=j$ then there is also a generically immersed arc $I \looparrowright A_i$.  
Similarly $(\partial V_{ij} \cup  \partial W_{ij}) \cap B_j$ is the image of a generic immersion  $\sqcup^{C_{ij}} S^1 \looparrowright B_j$, and if $i=j$  a generically immersed arc $I \looparrowright B_i$. Note that the number of circles in $A_i$ and $B_j$ is the same, so using $C_{ij}$ to denote both quantities is justified. 

\begin{lemma}\label{lemma:Q-path-connected}
    We can assume without loss of generality that $Q$ is path-connected.
\end{lemma}

\begin{proof}
 If $Q$ were not connected, there would exist an index $i$ such that the spheres
 $A_i \cup B_i$ are disjoint from $A_j \cup B_j$ for all $j \neq i$. We consider the corresponding $i$th eye in the Cerf graphic. 
By the independent trajectories principle of Hatcher-Wagoner~\cite{HW}*{\S 7}, there is a deformation of the family moving this eye away from all the others. Move this eye so it appears before all the others in the Cerf graphic, and then merge it with the outermost eye, similar to the deformation depicted in \cref{figure:merge}, but assuming the leftmost family has only one eye. Repeat this process while $Q$ remains disconnected. Since the process reduces the number of eyes and $Q$ is path-connected if the family has one eye, the algorithm terminates.  
 \end{proof}

\begin{remark}
    We could also have arranged for $Q$ to be path-connected by adding extra trivial finger and Whitney moves between $A_i$ and $B_j$ (or between $B_i$ and $A_j$). In keeping with the ethos of this article, we chose the method of proof of \cref{lemma:Q-path-connected} in order to minimize the complexity of $Q$, in particular to minimize $N$ in the next lemma.  
\end{remark}

\begin{lemma}\label{lemma:homotopy-type-of-Q}
Let $n$ denote the number of eyes, i.e.\ the number of spheres $A_i$ and the number of spheres $B_j$. Let $C_{ij}$, for $i,j = 1, \dots, n$, be the number of immersed circles corresponding to intersections of  $A_i$ with $B_j$, as above. 
Let $m := |\mathring{V} \pitchfork  \mathring{W}|$ denote the number of intersection points between the interiors of the $V$ discs and the interiors of the $W$ discs. Assume, using the proof of \cref{lemma:Q-path-connected} if necessary, that $Q$ is path connected. Then
\begin{equation} \label{eq: Q homotopy type}
    Q\simeq \bigvee^{N} S^2 \vee \bigvee^M S^1
\end{equation}
where 
\begin{equation} \label{eq:numbers}
N := 2n+\sum_{i,j} C_{ij} \text{ and } M := m + \sum_{i,j} C_{ij} - n + 1. 
\end{equation} 
\end{lemma}

\begin{proof}
To analyze the homotopy type of the Quinn core, we study its $2$-complex spine $\spineQ := A \cup B \cup V \cup W$, noting that since $Q$ is by definition a regular neighborhood of~$\spineQ$, we have that $\spineQ \simeq Q$.  Consider the construction of~$\spineQ$ as a three step process. We record a presentation of the fundamental group given at each step. First, the fundamental group of the union of the spheres $A\cup B$ is free, generated by double point loops, cf. \cite{FQ} or \cite{Freedman-notes}*{11.3}. 
Note that a single path is chosen from the base point in $M_{1/2}$ to each $2$-sphere, and the same collection of paths is used in the definition of all double point loops. Not all double points contribute generators of the fundamental groups, as some of them (one intersection point $A_i\cap B_i$ for each $i$, and one intersection point $(A_i\cup B_i)\cap(A_j\cup B_j)$ for $i\neq j$) reduce the number of connected components.
Let $p_1,\ldots, p_k$ be the remaining double points, and let $g_i$ denote the free generator corresponding to $p_i$.

Next, abstractly attach the finger and Whitney discs, with interiors disjoint from each other, to $A\cup B$. The attaching curve of each disc passes through exactly two double points, say $p_i$ and $p_j$, $i\neq j$. The effect of attaching this $2$-cell to $A\cup B$ on the fundamental group is the relation $g^{}_i g_j^{-1}=1$. Note that the boundaries of $V, W$ may intersect on $A$ and $B$, however, this is immaterial for writing down a presentation of the fundamental group. 

The result of attaching all finger and Whitney discs is a $2$-complex with free fundamental group. Finally, the spine of the Quinn core $\spineQ$ is obtained by introducing intersections between $V$ and $W$, giving rise to additional free generators $h_1,\ldots, h_m$. The resulting presentation is
\begin{equation} \label{eq:presentation1}
\pi_1(\spineQ)\cong \langle g_1,\ldots, g_k, h_1,\hdots, h_m \mid  g^{}_{i_1} g_{j_1}^{-1}, \ldots, g^{}_{i_\ell} g_{j_\ell}^{-1}, 1, \ldots, 1
\rangle \end{equation}
Here $\ell$ is the total number of finger and Whitney discs, and there are $2n$ trivial relations, corresponding to the spheres $A$ and $B$. Given two generators $g^{}_{i}, g_{j}$ and a relation $g^{}_{i} g_{j}^{-1}$, one of the generators, say $g_j$, and the relation may be removed from the presentation using Tietze moves.
Any other appearance of $g_j$ in a relation $g_jg_r^{-1}$ is replaced with $g_ig_r^{-1}$. Note that if $r=i$, this leads to the trivial relation $g_ig_i^{-1}$.
We refer to \cite{HAM} for a discussion of group presentations and $2$-complexes; to be specific the moves we used are all of the form (26) - (28) in that reference. They are called $Q^{**}$ transformations in \cite{HAM} and sometimes they are referred to as the Andrews-Curtis moves; they are all of the Tietze moves with the exception of adding a trivial relation. An inductive application of these moves reduces the presentation to 
\begin{equation} \label{eq:presentation2}
\pi_1(\spineQ)\cong \langle g_1,\ldots, g_{k'}, h_1,\hdots, h_m \mid 1, \ldots, 1
\rangle. 
\end{equation}
In addition to the $2n$ trivial relations in \cref{eq:presentation1}, a trivial relation $1$ (equal to $g_ig_i^{-1}$ for some $i$) appears as the result of Tietze moves for each immersed circle corresponding to intersections of $A_i$ with $B_j$, showing that there are a total of $N$ relations in \cref{eq:presentation2}.
The fact that the number of generators $k'+m$ equals $M$ may be deduced from the fact the Euler characteristic of $\spineQ$ equals $3n-m$, and it is unaffected by the Andrews-Curtis moves.
Note that \eqref{eq:presentation2} is also a presentation for the fundamental group of the $2$-complex on the right-hand side of \eqref{eq: Q homotopy type}. 
The equivalence classes of group presentations up to Andrews-Curtis moves are in bijective correspondence with $3$-deformation types of $2$-complexes, cf. 
\cite{HAM}*{Theorem 2.4}. Here a $3$-deformation refers to a simple homotopy equivalence given as a composition of elementary expansions and collapses through cells of dimensions at most $3$. It follows that $\spineQ$ and the $2$-complex on the right-hand side of \eqref{eq: Q homotopy type} are homotopy equivalent.
\end{proof}

The following condition will be important in our proofs. 
We see from \cref{lemma:homotopy-type-of-Q} that controlling the integers $C_{ij}$ allows us to control the homotopy type of $Q$. 

\begin{definition}\label{defn:arc-condition}
We say that \emph{Quinn's arc condition holds} if for each $i$ we have that $C_{ii}=0$, so there are no immersed circles corresponding to $A_i$, $B_i$ intersections, and if in addition for each $i$ both $(\partial V_{ii} \cup  \partial W_{ii}) \cap A_i \subseteq A_i$  and  $(\partial V_{ii} \cup  \partial W_{ii}) \cap B_i \subseteq B_i$ are embedded arcs. 
\end{definition}

In \cite{Quinn:isotopy}*{Section~4} Quinn proved the following lemma.  

\begin{lemma}[Quinn]\label{lemma:quinn-arc-condition}
There is a deformation such that Quinn's arc condition holds. 
\end{lemma}

Quinn's proof just does this for the innermost eye, but we can apply the proof to each of the eyes individually to obtain the same conclusion. 
This will likely create new $A_i$, $B_j$  intersections for $i \neq j$.

\begin{corollary}
    \label{lemma:homotopy-type-of-Q-single-eye-case-embedded-arcs}
    Suppose that there is a single eye and that Quinn's arc condition holds. 
    Let $m := |\mathring{V} \pitchfork  \mathring{W}|$ denote the number of intersection points between the interiors of the $V$ discs and the interiors of the $W$ discs. Then
    \begin{equation}\label{eq: homotopy equation}
    Q\simeq S^2 \vee S^2 \vee \bigvee^m S^1.
    \end{equation}
\end{corollary}

\begin{proof}
In this case, $n=1$ and $C_{ij}=0$ for all $i,j$, so the corollary follows from \cref{lemma:homotopy-type-of-Q}.
\end{proof}

\subsection{A handle decomposition} \label{sec: a handle decomposition}

The following result, cf.~\cite{CFHS}*{p.~344}, will be used in the proofs of \cref{thm:corks-for-PIs-body,thm:gen-barbell-cork-thm}. 
Let $S$ be a connected, compact, codimension zero submanifold in the interior of a compact 
simply-connected $4$-manifold $M$, such that $M\sm S$ is connected, and $\pi_1(S)$ is a free group of some rank $m$. Consider a handle structure ${\mathcal H}$ on $M\sm S$ without $0$-handles, relative to $\partial S$, and let $S':=S\cup (\text{all $1$-handles of }{\mathcal H})$. The fundamental group $\pi_1(S')$ is free, of rank $n$ equal to $m$ plus the number of $1$-handles of ${\mathcal H}$; let $g_1,\ldots, g_n$ be a set of free generators. Let  ${\mathcal H}'$ be the resulting handle decomposition of $M\sm S'$, relative to $\partial (M\sm S')$, consisting of the $2$-, $3$-, and $4$-handles of $\mathcal H$. 

\begin{lemma}\label{lemma:additional2handles} 
In the set-up described above, stabilize ${\mathcal H}'$ by introducing $n$ canceling $2$-, $3$-handles pairs in a 4-ball near the boundary of $S'$. 
After a sequence of $2$-handle slides, the attaching circles of the newly introduced 2-handles $h_1,\ldots, h_n$ may be assumed to represent the conjugacy classes of the free generators $g_1,\ldots, g_n$ of $\pi_1(S')$.
\end{lemma}

\begin{proof} 
Consider the presentation of the trivial group $\pi_1(M)$, given by the generators $g_1,\ldots,g_n$ and relations corresponding to the $2$-handles. (Here the basepoint is connected to the attaching circle of each $2$-handle by an arc; then the attaching circle gives rise to a relation, i.e.\ a word in the free group.) Since $\pi_1(M)=\{ 1\}$, each generator $g_i$ is in the normal closure of the relations, in other words, $g_i$ equals a product of conjugates of the relations. In other words, the trivial element, multiplied by a product of conjugates of the relations, equals $g_i$. Starting with the $i$th newly introduced $2$-handle $h_i$ (whose attaching circle is trivial), implement handle slides guided by the equation in the free group described in the preceding sentence. The result is the desired collection of $2$-handles.
\end{proof}

We will apply this lemma to $S = Q$ in order to augment $Q$ with 2-handles, and obtain a simply-connected submanifold of $M_{1/2}$ containing $Q$. We will also use the lemma in the proof of \cref{thm:Zm-in-a-diff-cork}.

\section{A cork theorem for 1-stably trivial exotic diffeomorphisms}\label{section:1-eyed-cork-thm}

In this section we prove the following theorem.
Let $X$ be a compact, simply-connected smooth 4-manifold, and let $F \colon X\times I \to X \times I$ be a smooth pseudo-isotopy. We consider $X \times I$ as a trivial $h$-cobordism from $X$ to itself. 

\begin{theorem}[Corks for one-eyed pseudo-isotopies]\label{thm:corks-for-PIs-body}
    If $F$ admits a Cerf family with one eye, then there exists a compact, contractible, codimension zero, submanifold $C \subseteq X$ and a smooth isotopy of $F$ rel.\ $X \times \{0\} \cup \partial X \times I$ to a pseudo-isotopy $F' \colon X \times I \to X \times I$ that is supported on $C \times I$. 
\end{theorem}

\begin{proof}[Proof of \cref{thm:corks-for-PIs-body}] 
The proof is inspired by the proof of the cork theorem for $h$-cobordisms \cites{CFHS, Matveyev}; see also the exposition in \cite{Kirby_corks}. We start with the submanifold $Q$ in the middle-middle level $M_{1/2} \cong X \# (S^{2} \times S^{2})$, defined in \eqref{eq:Q}.
By \cref{lemma:quinn-arc-condition} we may assume that $V$ and $W$ satisfy Quinn's arc condition. 
Recall that the finger discs~$V$ are disjointly embedded in $X \# S^{2} \times S^{2}$, and so are the Whitney discs $W$. If the interiors of $V, W$ are disjoint, $Q$ is simply-connected. In general, the interiors of $V, W$ intersect, and the homotopy type of the Quinn core is given in \cref{lemma:homotopy-type-of-Q-single-eye-case-embedded-arcs} to be $S^2 \vee S^2 \vee  \bigvee^m S^1$, where $m$ is the number of intersections between the interiors of the $V$ and $W$ discs. 
Next we apply the construction of \cref{sec: a handle decomposition} to $S=Q$, obtaining $S'=Q\, \cup 1$-handles, with $S'\simeq S^2\vee S^2\vee \bigvee^{m'} S^1$ for some $m'\geq m$.
%In particular, $\pi_1(S')$ is a free group.  
Let $\{\gamma_i\}$ be a collection of loops in $S'$ corresponding to the $S^1$ wedge summands. Applying  \cref{lemma:additional2handles} to $S'$ we find a collection of $2$-handles $\{H_i\}_{i=1}^m$ in $M_{1/2} \sm S'$ with attaching curves homotopic to the $\{\gamma_i\}$. Define the submanifold \[R:=S'\cup_{i=1}^{m'} H_i.\] During the process of adding extra 2-handles, no new second homology is introduced and all of the generators of $\pi_1(S')$ are canceled. Hence it follows from \cref{lemma:homotopy-type-of-Q-single-eye-case-embedded-arcs} that $R$ is simply-connected and $R \simeq S^2 \vee S^2$.  
Flowing along $v_{1/2}$ downwards surgers $R$ along the sphere  $A$ and flowing upwards surgers $R$ along the sphere $B$. This gives a simply-connected cobordism $U \subseteq X \times I$. 
The cobordism $U$ is obtained from $R \simeq S^2 \times S^2$ by attaching two 5-dimensional 3-handles, homotopically attaching one 3-handle to each of the $S^2$ wedge summands. Hence $U \simeq D^3 \vee D^3 \simeq \{\ast\}$, i.e.\ $U$ is contractible. Note that $U$ contains the critical points of $q_{1/2}$ by construction and the critical points are algebraically canceling, so $U$ is an $h$-cobordism. Hence $C := U \cap (X \times \{0\})$ is contractible.  
In fact, $U$ is a trivial $h$-cobordism, because using the Whitney discs $W$ (or the $V$) we can arrange that all critical points are in geometrically canceling position. Hence $C =  U \cap (X \times \{0\}) \cong U \cap (X \times \{1\})$.   

By construction, $Q \subseteq U$, and $Q_0 \subseteq C \subseteq X \times \{0\}$.  We can therefore apply \cref{lem:Quinn Core Lemma} to obtain a smooth isotopy from $F$ to $F'$ such that $F' = \Id$ on $(X \sm \mathring{C}) \times I$.  We have succeeded in decomposing $X \times I$ into $(C \times I) \cup ((X \sm \mathring{C}) \times I)$, and isotoping $F$ to $F'$, such that $F'$ is supported on the contractible piece $C \times I$.  
\end{proof}

As a consequence of \cref{thm:corks-for-PIs-body} along with \cref{lemma:n-stable-isotopy-iff-n-eyed-PI} we deduce \cref{thm:one-eye-thm}. It is the special case $m=1$ of the following more general theorem.  

%\oneeyethm*

\begin{theorem}[Diffeomorphism cork theorem, version for finite collections]
\label{thm:one-eye-thm-full-version}
Let $X$ be a compact, simply-connected, smooth 4-manifold, and let $\{f_i\}_{i=1}^m$ be a collection of boundary-fixing diffeomorphisms of $X$ such that $f_i$ is 1-stably isotopic to $\Id$ for each $i$.  
    Then there exists a compact, contractible, codimension zero, smooth submanifold $C \subseteq X$, and for each $i=1,\dots,m$ there is a  boundary fixing isotopy of $f_i$ to a  diffeomorphism $f'_i \colon X \to X$ such that $f_i'$ is supported on $C$.  

     Moreover, $C$ can be chosen to be a 4-manifold that admits a handle decomposition into 0-, 1-, and 2-handles. 
\end{theorem}

\begin{proof}
 By \cref{lemma:n-stable-isotopy-iff-n-eyed-PI}, for each $i$ there is a pseudo-isotopy $F_i \colon X \times I \to X \times I$  from $f_i$ to the identity of $X$ that admits a Cerf family with one eye. 
  We have that $F_i|_{X \times \{ 0\}} = \Id$ and $F_i|_{X \times \{ 1\}} = f_i$. 
  By \cref{thm:corks-for-PIs-body}, there is a compact, contractible submanifold $C_i \subseteq X$ such that $F_i$ is isotopic rel.\ boundary to $F_i'$, where $F_i' = \id$ on $X \sm \mathring{C}_i \times I$. Restricting this isotopy to $X \times \{1\}$ yields an isotopy from $f_i$ to $f_i'$ such that $f_i' = \id$ on $(X \sm \mathring{C}_i) \times \{1\}$. 

Next we show that for each $i$, $C_i$ can be constructed from $0$-, $1$-, and $2$-handles. Our proof is analogous to Matveyev's proof of the analogous fact for corks of $h$-cobordisms.  
  The starting point is a handle decomposition of the Quinn core $Q$
  which has handles of index at most 2. A detailed analysis of the Kirby diagram of a regular neighborhood of $A\cup B\cup W$ is given in  \cite{Matveyev}*{Figures 3-6}; see also \cref{figure:Whitney_move} above. In the present context there is a second collection $V$ of Whitney discs; it is incorporated in a Kirby diagram analogously to $W$. A key feature of the Kirby diagram for this handle decomposition is that the 0-framed unknotted 2-handle corresponding to $A$, taken together with the dotted components corresponding to all $1$-handles, forms an unlink. 
The 4-manifold $C_i$ is obtained by surgering $A$, and hence a Kirby diagram for $C_i$ is obtained by replacing the 0-framed 2-handle by a dotted circle. We obtain a handle decomposition of~$C_i$ with only $0$-, $1$-, and $2$-handles, as desired. 

Next we show that all of the $f_i'$ can be isotoped so as to be supported on a single diff-cork. We use that each $C_i$ is built out of $0$-, $1$-, and $2$-handles. By transversality we may assume that the $C_i$ only intersect in their 2-handles.
To see this note that an ambient isotopy $h_t \colon X \to X$ of the support $C_i$ of a diffeomorphism $f_i'$ can be realized by an isotopy $h_t \circ f_i' \circ h_t^{-1}$ of the diffeomorphism.  
Hence we may assume that $\bigcup_{i=1}^m C_i \simeq \bigvee^k S^1$, for some~$k$.  Note that $\bigcup_{i=1}^m C_i$ still admits a handle decomposition with $0$-, $1$-, and $2$-handles only. 

Apply the method of \cref{sec: a handle decomposition} and \cref{lemma:additional2handles}  to $S:=\bigcup_{i=1}^m C_i$. That is, consider $S'=S\cup 1$-handles, so that $S'\simeq \bigvee^\ell S^1$ for some $\ell\geq k$. Then
find a collection of 2-handles of $X \sm S'$, precisely canceling free generators of $\pi_1(\bigvee^l S^1) \cong \pi_1(S')$. The union $C := S'\, \cup\,  2\text{-handles}$ is a contractible manifold, by the same argument as in the proof of \cref{thm:corks-for-PIs-body}.  Since each $C_i$ is contained in $C$, each diffeomorphism $f_i'$ is supported on $C$, as asserted. 
\end{proof}

\section{The sum square move and the associated \texorpdfstring{$\pi_2(X)$}{second homotopy group} element}\label{section:sum-square-move-and-pi2-element}

In this section, in preparation for the proof of \cref{thm:gen-barbell-cork-thm}, we recall Quinn's sum square move, and we note its effect on certain elements of $\pi_2(X)$ determined by a Cerf family.

\subsection{The sum square move}\label{subsection:sum-square-move}

Quinn's sum square move \cite{Quinn:isotopy}*{Section~4.2} gives rise to a deformation of a pseudo-isotopy. 
We consider a 1-parameter family in nested eye form. In the middle-middle level $X \#^n (S^2 \times S^2)$ we have the data of two collections of embedded spheres $A$ and $B$, which intersect each other transversely. We have finger discs $V$ and Whitney discs $W$, such that each collection of discs cancels all the excess intersections between the $A$ and $B$ spheres. 

We describe the sum square move in the middle-middle level.
Quinn~\cite{Quinn:isotopy}*{Section~4.2}  justified why the sum square move gives a deformation of the pseudo-isotopy. 
The move alters either the finger or Whitney discs, and their boundaries. We will explain the version that alters the Whitney discs.  

To implement the sum square move, we need a framed embedded square~$S$ in the middle-middle level, with the interior of $S$ disjoint from $A \cup B \cup W$. 
The square must have two edges on two distinct $W$ discs (labeled $W_1$ and $W_2$ in Figure \ref{figure:sum_square}), one edge on $A$, and one on $B$. New $W$ discs are obtained by cutting $W_1$ and $W_2$ 
along the boundary edges of the sum square $S$, and gluing in two parallel copies of $S$. In one possible arrangement, the effect of the move on the boundaries of the discs, on $A$ and~$B$ spheres, is illustrated in Figure \ref{figure:sum_square_boundaries}. Note that $+$ and $-$ intersection points are still paired up after the move.  

\begin{figure}[ht]
\centering
\includegraphics[height=3.2cm]{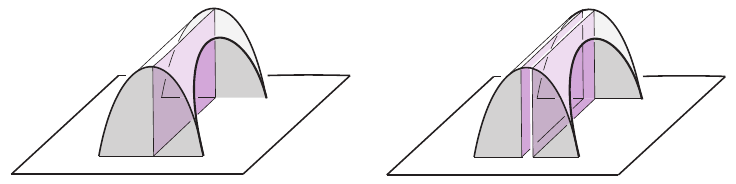}
{\scriptsize
\put(-257,75){$B$}
\put(-250,15){$A$}
\put(-305,30){$W_1$}
\put(-252,50){$W_2$}
\put(-266,50){$S$}
}
\caption{The sum square move along the sum square $S$ shown in purple.}
\label{figure:sum_square}
\end{figure}

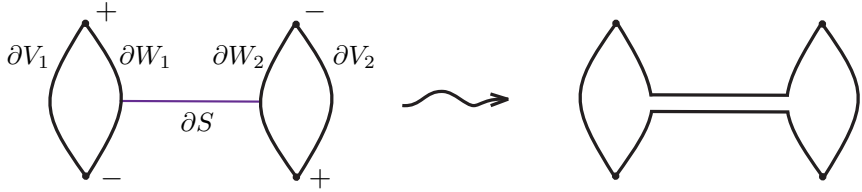
\begin{figure}[ht]
\centering
\includegraphics[height=2.3cm]{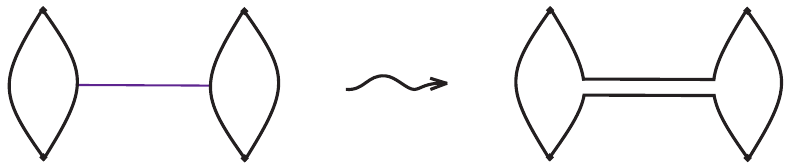}
\put(-290,63){$+$}
\put(-288,0){$-$}
\put(-212,63){$-$}
\put(-210, 0){$+$}
{\small 
\put(-281,46){$\partial W_1$}
\put(-258,21){$\partial S$}
\put(-323,46){$\partial V_1$}
\put(-246,46){$\partial W_2$}
\put(-201,46){$\partial V_2$}
}
\caption{Rearranging the boundaries of the finger and Whitney discs, on $A$ and on $B$, using the sum square move.}
\label{figure:sum_square_boundaries}
\end{figure}

Figure \ref{figure:sum_square}, closely following Quinn's figure in \cite{Quinn:isotopy}*{Section~4.2}, depicts a $3$-dimensional model for the sum square. Here we see $A$, $W_1$, and  $W_2$, together with a neighborhood of the arc of $\partial S$ in $B$ lie in  ${\mathbb R}^3\times\{ 0\}\subseteq {\mathbb R}^3\times {\mathbb R}$. The rest of $B$ extends into the past and the future. The framing of $S$ along its boundary is determined in the $3$-dimensional model by a non-vanishing vector field on $\partial S$ which is normal to $S$ and tangent to $A$, $B$, and the $W$ discs.  This framing has to admit an extension over $S$ for the move to yield embedded $W$ discs.

In applications, one has to work to find a sum square $S$ satisfying the conditions laid out above. Making use of dual spheres, which  one can always find in a pseudo-isotopy, Quinn~\cite{Quinn:isotopy}*{Section~4.2} shows how to obtain a sum square produced from an arbitrary choice of null-homotopy of the boundary square. From this one can produce an embedded sum square that is framed and whose interior is disjoint from $A \cup B \cup W$ as desired.

\subsection{Creating a single circle}\label{subsection:creating a single circle}

Here is an initial use of the sum square move. 
Consider two spheres $A_i$ and $B_j$ with $i \neq j$ fixed. We write $V_{ij}$ (respectively $W_{ij}$) for the collection of finger (respectively Whitney) discs pairing intersections of $A_i$ and $B_j$. The algebraic intersection number $\lambda(A_i,B_j)$ vanishes, since $i \neq j$. It follows, assuming genericity, that the intersection $(\partial V_{ij} \cup  \partial W_{ij}) \cap A_i$  (and similarly $(\partial V_{ij} \cup  \partial W_{ij}) \cap B_j$) is  the image of a generic immersion $\sqcup^k S^1 \looparrowright A_i$ (respectively $\sqcup^k S^1 \looparrowright B_j$), for some $k \geq 0$.

\begin{lemma}\label{lemma:assume-one-circle-using-sum-square}
After a deformation of the family, we can arrange that $k=1$. 
 \end{lemma}

\begin{proof}
    Consider two generically immersed circles $\gamma_1$ and $\gamma_2$ on $A_i$, which are a subset of $(\partial V_{ij} \cup  \partial W_{ij}) \cap A_i$.  Each of these circles $\gamma_\ell$, for $\ell \in \{1,2\}$, comprises a union of disjointly embedded arcs $\gamma_\ell^V$ from $\partial V_{ij}$ and a union of disjointly embedded arcs $\gamma_\ell^W$ from $\partial W_{ij}$.  We also have that $\gamma_1^V \cap \gamma_2^V = \emptyset$ and $\gamma_1^W \cap \gamma_2^W = \emptyset$.  There can be an uncontrolled number of intersections between $\gamma_i^V$ and $\gamma_j^W$, for each nonempty subset $\{i,j\} \subseteq \{1,2\}$.  

By taking the union of the finger and Whitney discs corresponding to $\gamma_1 \cup \gamma_2$, and considering their intersection with $B_j$, we have an analogous situation on $B_j$, consisting of generically immersed circles $\delta_1$ and $\delta_2$, expressed as a union of arcs $\delta_\ell = \delta_\ell^V \cup \delta_\ell^W$.   
    
We will show how to combine $\gamma_1$ and $\gamma_2$ into a single circle using the sum square move. 
Since $\partial W_{ij} \cap A_i$ is a disjoint union of embedded arcs in $A_i$, we have that $A_i \sm (\partial W_{ij} \cap A_i)$ is path connected. Hence we can join a point in the interior of one arc in $\gamma_1^W$, to a point in the interior of an arc in $\gamma_2^W$, via a smoothly embedded arc $\sigma_A$ in $A_i$ whose interior lies in $A_i \sm \partial W$, and which abuts to $\gamma_1^W \cup \gamma_2^W$ transversely. 
Similarly, we can find an arc $\sigma_B$ on $B_j$, joining the other boundary arcs of the same pair of Whitney discs. 

These arcs $\sigma_A$ and $\sigma_B$ form two sides of the boundary of a sum square. By \cref{subsection:sum-square-move}, or \cite{Quinn:isotopy}*{Section~4.2}, we can complete this to a sum square $S$ that is framed and embedded with interior disjoint from $A \cup B \cup W$. 
We choose the arcs $\sigma_A$ and $\sigma_B$ in such a way that after the sum square move, we obtain Whitney discs pairing
$+$ and $-$ double points of $A_i \cap B_j$.   Then performing the sum square move yields a deformation of the family to one where $(\partial V_{ij} \cup  \partial W_{ij}) \cap A_i$  and $(\partial V_{ij} \cup  \partial W_{ij}) \cap B_j$ both consist of one fewer circles than before the move. Iterating the procedure we reduce to the case of a unique circle, i.e.~$k=1$. 
\end{proof}

\subsection{An element of \texorpdfstring{$\pi_2(X)$}{the second homotopy group} associated to each circle}\label{subsection:element-pi2}

We describe the $\pi_2$ elements associated to the circles $\gamma_i$ as described in \cref{subsection:creating a single circle} and we show how they add when we do the sum square to combine two circles. 

Each disc $U_k$ in the union $V \cup W$ has an arc $\partial_A U_k$ that lies on an $A$-sphere and an arc $\partial_B U_k$ on a $B$ sphere. The arcs $\partial_A U_k$ and $\partial_B U_k$ intersect at their endpoints, which lie in~$A \pitchfork B$.

Suppose we have a collection of discs $\{U_{k_1},\dots,U_{k_m}\}$ taken from the finger and Whitney discs $\{U_k\}$, such that \[\gamma_A:= \bigcup_{\ell=1}^m \partial_A U_{k_\ell}\] is an immersed circle $\gamma_A$ that lies on some sphere in the middle-middle level, $A_i \subseteq M_{1/2} = X \#^n (S^2 \times S^2)$.  
Recall that the arcs on $A_i$ coming from $V$ discs are mutually disjoint, as are the arcs coming from $W$ discs. However, the two collections of arcs may meet on $A_i$.  The $U_{k_\ell}$ all pair up intersections with the same $B$ sphere, $B_j$ say, and the union \[\gamma_B:= \bigcup_{\ell=1}^m \partial_B U_{k_\ell}\] is an immersed circle $\gamma_B \subseteq B_j$. 

Using that $A_i$ is simply-connected, choose a null-homotopy \[\Delta_A \colon D^2 \to A_i\] for $\gamma_A$, and a null-homotopy \[\Delta_B \colon D^2 \to B_j\] for $\gamma_B$. We consider a map 
\[\psi_{i,j} \colon S^2 \to X\#^n (S^2 \times S^2)\]
of a 2-sphere in the middle-middle level, obtained by gluing $\Delta_A \colon D^2 \to A_i$ and $\Delta_B \colon D^2 \to B_j$ to the union $\bigcup_{\ell=1}^m  U_{k_\ell}$, as described in Figure \ref{figure:Belted_sphere}.

Let us describe this in more detail. We take a union $\sqcup^m D^2$ corresponding to $\{U_{k_\ell}\}_{\ell=1}^m$, identify $(0,1)$ in the $\ell$th disc with $(-1,0)$ in the $(\ell+1)$st, for $\ell=1,\dots,m$ and do the same with the $m$th and the $1$st. Then we embed this quotient space $\sqcup^m D^2/\sim$ around the equator of $S^2$ by a map $E \colon \sqcup^m D^2 \to S^2$, in such a way that the complement of the image consists of two open discs. We define the map $\psi_{i,j}$ on the image of $E$ by the composite: first send $E(x)$ to $x \in \sqcup^m D^2$ (or some other $y \in E^{-1}(E(x))$), and then use the map \[\sqcup^m D^2 \to (\sqcup^m D^2/\sim) \to \bigcup_{\ell =1}^m U_{k_\ell} \subseteq X\#^n (S^2 \times S^2),\] 
where the last map sends the $\ell$th disc to $U_{k_\ell}$. 
We extend the map $\psi_{i,j}$ to all of $S^2$ using~$\Delta_A$ and~$\Delta_B$. This completes the description of the map $\psi_{i,j}$. 

\begin{figure}[ht]
\centering
\includegraphics[height=3.8cm]{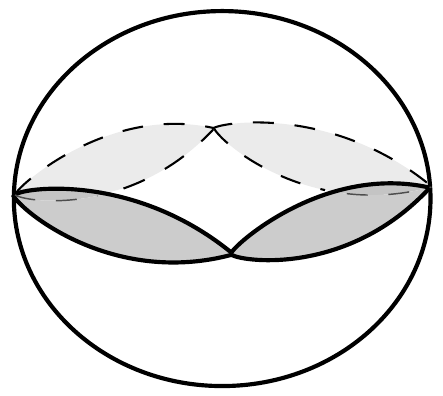}
{\small
\put(-67,86){$\Delta_A$}
\put(-67,16){$\Delta_B$}
}
\caption{$S^2=\cup_{\ell=1}^m U_{k_l}\cup \Delta_A\cup \Delta_B$. The discs $U_{k_\ell}$ are shaded.}
\label{figure:Belted_sphere}
\end{figure}

Note that there are many choices for the null-homotopies $\Delta_A$ and $\Delta_B$, and altering the choice made can change the homotopy class of $\psi_{i,j}$. 

Now consider the degree one map 
\[J \colon X \#^n (S^2 \times S^2) \to X\]
defined by sending $\#^n (S^2 \times S^2) \sm \mathring{D}^4 \to D^4$. The induced map $J_* \colon \pi_2(X \#^n (S^2 \times S^2)) \to \pi_2(X)$ has the effect of sending the homotopy classes that are supported in $\#^n (S^2 \times S^2)$ to $0$.  Define
\[\theta_{i,j} := J \circ \psi_{i,j} \colon  S^2 \to X.\]
To fully determine this map we need to fix an orientation convention. Fix once and for all an orientation of $S^2$, and as before use the orientations of each sphere $A_i$ and $B_j$, such that the intersection numbers $\lambda(A_i,B_i) = +1$ for each $i$. Each finger or Whitney disc, pairs two double points, one with $+$ intersection sign and one with $-$ intersection sign. We fix an orientation of each finger or Whitney disc $U_k$ by orienting the tangent space at the $+$ intersection point, denoted $p_k^+$. For the first tangent vector, choose a nonzero vector in $T_{p_k^+}\partial_A U_k \subseteq T_{p_k^+}A$, pointing into the interior of $\partial_A U_k$. For the second tangent vector, choose a nonzero vector in $T_{p_k^+}\partial_B U_k \subseteq T_{p_k^+}B$, pointing into the interior of $\partial_B U_k$. This determines an orientation of $TU_k$. These orientations are consistent for each disc $U_{k_{\ell}}$ used in the construction of the map $\psi_{i,j}$, and hence we fix $\theta_{i,j}$ on the nose, and not just up to sign.

\begin{lemma}
The homotopy class of $\theta_{i,j} \in \pi_2(X)$ is independent of the choice of $\Delta_A$ and~$\Delta_B$.    
\end{lemma}
  \begin{proof}
      The difference in any two choices for $\Delta_A$ represents a multiple of the class of $[A] \in \pi_2(X \#^n (S^2 \times S^2))$, 
      and similarly for $\Delta_B$ and $B$.    However $[A]$ and $[B]$ belong to $\ker J_*$, so the image $J_*(\psi_{i,j})$ is unaffected by the choices. 
  \end{proof}

Now we suppose that there are two circles $\gamma_A^1$ and $\gamma_A^2$ on $A_i$ and two circles $\gamma_B^1$ and~$\gamma_B^2$ on~$B_j$ corresponding to the boundaries of distinct collections of discs in $V \cup W$. Assume that $\gamma_A^1$ and $\gamma_B^1$ cobound a collection of discs, $\{U_{k_1}^1,\dots,U_{k_m}^1\}$, and similarly $\gamma_A^2$ and $\gamma_B^2$ cobound a collection of discs $\{U_{k_1}^2,\dots,U_{k_p}^2\}$.   We also choose null-homotopies $\Delta_A^1 \colon D^2 \to A_i$ for $\gamma_A^1 \subseteq A_i$, $\Delta_B^1 \colon D^2 \to B_j$ for $\gamma_B^1 \subseteq B_j$, $\Delta_A^2 \colon D^2 \to A_i$ for $\gamma_A^2 \subseteq A_i$, and $\Delta_B^2 \colon D^2 \to B_j$ for $\gamma_B^2 \subseteq B_j$. 
Let $\theta_{i,j,1}$ and $\theta_{i,j,2}$ denote the resulting elements of $\pi_2(X)$, with $\theta_{i,j,1}$ corresponding to  $\gamma_A^1$  and  $\gamma_B^1$, and $\theta_{i,j,2}$ corresponding to  $\gamma_A^2$  and  $\gamma_B^2$. 

Without loss of generality, suppose that we do the sum square move using a Whitney disc $W_1$ in the first collection of discs, and a Whitney disc $W_2$ in the second collection. Let $\theta_{i,j} \in \pi_2(X)$ denote the element corresponding to the single circle of finger/Whitney arcs (on each of $A_i$ and $B_j$) that arises after the sum square move.

\begin{lemma}\label{lemma:pi_2-elements-adding}
    We have that $\theta_{i,j} = \theta_{i,j,1} + \theta_{i,j,2} \in \pi_2(X)$. 
\end{lemma}

\begin{proof}
We assume the null-homotopies $\Delta_A^i$ and $\Delta_B^i$ are chosen such that, near the boundary, and with respect to the model sum square in \cref{figure:sum_square}, they move away from $W_1$ and $W_2$ in the opposite direction to $S$. This can be arranged by changing the choice of null-homotopies for $\gamma_A^k$ and $\gamma_B^k$, $k = 1,2$, if necessary. 

The sum square move removes a square from the disc $W_1$ (a neighborhood of the close edge in \cref{figure:sum_square}, the one that lies in $W_1$, and removes a square from the disc $W_2$ (a neighborhood of the far edge in \cref{figure:sum_square}, the edge that lies in $W_2$). The sum square move also combines the null-homotopies $\Delta_A^1$ and $\Delta_A^2$ with a strip in $A_i$ (near the lower edge of the sum square in \cref{figure:sum_square}), and it combines the null-homotopies $\Delta_B^1$ and $\Delta_B^2$ with a strip in $B_j$ (near the upper edge of the sum square in \cref{figure:sum_square}).  The union of these two strips and two copies of the sum square form a tube (with square cross section), $\partial I^2 \times I$. The result of the sum square is therefore exactly to perform an ambient connected sum of the two 2-spheres representing $\theta_{i,j,1}$ and $\theta_{i,j,2}$. 

A careful analysis of the orientation convention and the model sum square move in \cref{figure:sum_square} yields that the classes add (rather than taking their difference). 
\end{proof}

\begin{remark}
    If $X$ were not simply-connected, then a proof of \cref{lemma:pi_2-elements-adding} would presumably need a careful analysis of basing paths. 
\end{remark}
 
\section{Many-eyed diffeomorphisms are supported in a null-homotopic homotopy wedge of 2-spheres}\label{section:many-eyed-theorem}

When $f \colon X \to X$ must be stabilized by more than one copy of $S^2 \times S^2$ in order to smoothly trivialize it, equivalently when all pseudo-isotopies for $f$ must have more than one eye, we do not know how to find a contractible diff-cork, but we can prove the following theorem.

\genbarbellthm*

\begin{remark}
    Kronheimer-Mrowka~\cite{Kronheimer-Mrowka-K3} showed that the Dehn twist on the connect-sum $S^3$ in $K_3 \# K_3$, denoted  $D \colon K_3 \# K_3 \to K_3 \# K_3$, is not isotopic to the identity, and Jianfeng Lin~\cite{Lin-dehn-twist-stabn} showed that this continues to hold even after one stabilization by $S^2 \times S^2$. This diffeomorphism is stably isotopic to the identity by~\cites{Quinn:isotopy,gabai20223spheres} and is topologically isotopic to the identity by work of Kreck, Perron, and Quinn~\cites{Kreck-isotopy-classes,Perron,Quinn:isotopy,GGHKP} (see \cref{thm:TFAE-isotopy-notations}).  But we cannot apply \cref{thm:one-eye-thm} to this diffeomorphism, whereas \cref{thm:gen-barbell-cork-thm} does apply to it.  

Here is a related observation. Suppose that $D$ becomes isotopic to the identity after connected summ with $n$ copies of $S^2 \times S^2$ (we are not sure what the minimal $n$ is). Then $D \# \Id \colon K_3 \# K_3 \#^{n-1} (S^2 \times S^2)$ admits a diff-cork. It is not clear that this diff-cork can be isotoped into the $K_3 \# K_3$ summands. 
\end{remark}

The following statement will be used in the proof of \cref{thm:gen-barbell-cork-thm};  here we use the notation $\theta_{k,\ell}\in\pi_2(X)$ introduced in \cref{subsection:element-pi2}.

\begin{lemma}\label{lemma:trivial-PI-with prescribed-pi2-element}
    Let $x \in \pi_2(X)$, let $n \in \mathbb{N}_0$, and fix $i \neq j$ in $\{1,\dots,n\}$. There exists a trivial pseudo-isotopy  with $n$ eyes such that $\theta_{i,j} =x$ and $A_k \cap B_\ell = \emptyset$ 
    $($and hence $\theta_{k,\ell}=0)$
    for  $(k,\ell) \neq (i,j)$ and $k \neq \ell$.  We also require that $A_k \pitchfork B_k$ is exactly one point, for each $k$. Moreover, we can assume that in the middle-middle level $|A_i \pitchfork B_j| =3$, and the boundaries of the $V$ and $W$ discs form a circle on $A_i$ and a circle on $B_j$.  
\end{lemma}

\begin{proof}
    We start by constructing a particular pseudo-isotopy satisfying the conditions of \cref{lemma:trivial-PI-with prescribed-pi2-element}. First consider a trivial pseudo-isotopy with $n$ eyes, where $n$ pairs $A_i, B_i$ are born and then die. No intersections between $A_k$, $B_l$ occur in this family, for any $k,\ell$. For the given indices $i, j$ introduce a finger move between $A_i, B_j$ and immediately reverse it with a Whitney move where the Whitney disc $W$ is a parallel copy of the finger disc $V$. 

    It is not immediately clear that the following modification is a deformation of the trivial pseudo-isotopy. This will be justified at the end of the proof.
    Consider an immersed sphere $S$ representing the element $x\in \pi_2(X)$, and implement an ambient connected sum of $W$ with $S$ along an embedded arc connecting them. We continue denoting by $W$ the resulting disc. Use boundary twisting, cf. \cite{FQ}*{Section 1.3} or \cite{Freedman-notes}*{Section 15.2.2}, to correct the framing of $W$ while introducing additional intersections between $W$ and $A$ or $B$.
    Recall from \cite{Quinn:isotopy}*{Section~4.3}
    that $A$ and $B$ have duals (framed, embedded geometrically dual spheres) that are disjoint from $W$ but intersect $V$. Use these duals to make $W$ embedded and its interior disjoint from $A \cup B$, while preserving the framing of $W$. These $V,W$ determine the desired pseudo-isotopy. Note that it satisfies the requirements of the lemma by construction, except that we still need to justify that this pseudo-isotopy is trivial. 
    %{\footnote{AM- Can we carefully make the interior of $V$ and $W$ disjoint? Since $A\cup B\cup V\cup W$ is a 2-complex, we can choose an arc disjoint from all and connecting $W$ and $S$. SK: The arc is certainly disjoint but $V$ and $W$ intersect because of the properties of dual spheres that were used. }}
    
    To show the triviality of the pseudo-isotopy, we cancel the eyes, innermost first, working outwards. There are no obstructions to doing so; indeed, there are no extra intersections between $A_k$ and $B_k$ for any index $k$. 
    Suppose $i<j$, so the $i$-th eye is located in the interior of the $j$-th one. 
    Closing the $A_i$, $B_i$ eye certainly does not create any extra  intersections between $A_k$ and $B_k$ for any $k>i$, $k\neq j$, because $(A_i\cup B_i)\cap (A_k\cup B_k)=\emptyset$. 
    
    Next observe that this does not create any intersections between $A_j$ and $B_j$ either. A crucial ingredient here is the fact that while $A_i$ intersects $B_j$ in the pseudo-isotopy constructed above, $A_j$ is disjoint from $B_i$. 
    Canceling the $i$-th eye corresponds to a deformation of the $1$-parameter family of gradient-like vector fields. Consider the restriction of this deformation to the middle level $t=1/2$.
    Within the Cerf graphic at $t=1/2$, consider the $5$-dimensional cobordism supported in a neighborhood of $A_i\cup B_i$: it is a $5$-ball $D^5$ obtained from a neighborhood of $A_i\cup B_i$ by attaching two $5$-dimensional $3$-handles. 
    A deformation of the gradient-like vector field canceling the $i$-th eye is supported in this $5$-ball. While there are flow lines of the gradient-like vector field connecting $B_j$ and $D^5$, there are no such flow lines for $A_j$. 
    It follows that after the deformation, no new intersection points are created in the middle-middle level between $A_j$ and $B_j$. This shows that the constructed pseudo-isotopy is trivial, concluding the proof of the lemma.
\end{proof}

\begin{proof}[Proof of \cref{thm:gen-barbell-cork-thm}]
As in the proof of \cref{thm:one-eye-thm} in \cref{section:1-eyed-cork-thm}, the starting point is the fact that $f$ is pseudo-isotopic to $\id$ by \cref{thm:TFAE-isotopy-notations}. Using \cref{lemma:n-stable-isotopy-iff-n-eyed-PI}, for the remainder of the proof we will work with a pseudo-isotopy $F$ admitting a Cerf graphic with precisely $n$ eyes.

Apply \cref{lemma:quinn-arc-condition} to arrange that each eye satisfies Quinn's arc condition, and apply \cref{lemma:assume-one-circle-using-sum-square} to arrange that for each $1 \leq i\neq j \leq n$  the boundary arcs of the finger and Whitney discs pairing up intersections between $A_i$ and $B_j$ form a single circle on each of $A_i$ and $B_j$. 

We consider the elements $\theta_{i,j} \in \pi_2(X)$ introduced in \cref{subsection:element-pi2}. 
Given a pair $(i,j)$ with~$i \neq j$ and $\theta_{i,j} \neq 0$, use \cref{lemma:trivial-PI-with prescribed-pi2-element} to construct a trivial pseudo-isotopy with $x:=-\theta_{i,j}\in\pi_2(X)$ and $\theta_{k,\ell}=0$ for any $(k,\ell)\neq (i,j)$, $k\neq \ell$. 
Concatenate it with the given pseudo-isotopy using the ``uniqueness of birth'' move from \cite{HW}*{Chapter V} or \cite{Cerf}*{Chapter III}. As illustrated on the right in \cref{figure:merge}, after this merge there are still two instances of finger and Whitney move times: $t_f$, $t_w$ corresponding to the original pseudo-isotopy, and $t'_f$, $t'_w$ for the constructed one. 
By general position, finger move times~$t'_f$ can be pushed to the left in the Cerf diagram and Whitney move times $t_w$ to the right. The result is a single time when finger moves take place, shortly after the birth of all $A,B$ spheres, and a single Whitney move time shortly before their death. By \cref{lemma:assume-one-circle-using-sum-square}, we can again use the sum square move to arrange that the boundaries of finger and Whitney discs form at most a single circle in $A_i$ and in $B_j$, i.e.\ $C_{ij} \leq 1$. Moreover, by \cref{lemma:pi_2-elements-adding} we have arranged that $\theta_{i,j} =0$.
  
Implementing this step sequentially, we achieve
$\theta_{i,j} =0$ and $C_{ij} \leq 1$ for all pairs $(i,j)$ with $i\neq j$.

\begin{figure}[ht]
\centering
\includegraphics[height=2cm]{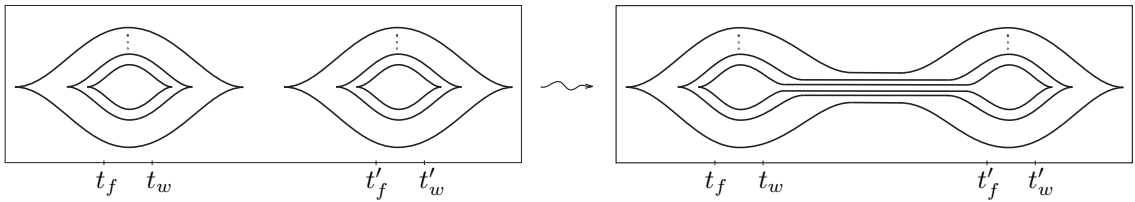}
\scriptsize{
\put(-340,-7){$t_f$}
\put(-323,-7){$t_w$}
\put(-250,-7){$t'_f$}
\put(-235,-7){$t'_w$}
\put(-141,-7){$t_f$}
\put(-126,-7){$t_w$}
\put(-53,-7){$t'_f$}
\put(-37,-7){$t'_w$}
}
\caption{Concatenation of pseudo-isotopies}
\label{figure:merge}
\end{figure}

Recalling that $n$ denotes the number of eyes, the middle-middle level contains the spheres $A_i, B_i$, for $1\leq i\leq n$.  The homotopy type of the Quinn core $Q$ was determined in \cref{lemma:homotopy-type-of-Q}. 

An application of \cref{lemma:additional2handles} shows that $1$- and $2$-handles may be added to $Q$ in $M_{1/2}$ to make the resulting submanifold $R$ simply-connected. 
Moreover, $H_2(R)\cong H_2(Q)$ since each new $2$-handle produced in the proof of \cref{lemma:additional2handles} cancels a corresponding generator of the free fundamental group.

The spine of $R$ is obtained from $\spineQ$ by adding an $S^1$ wedge summand for each new $1$-handle of $R$, and then adding a $2$-cell to each $S^1$ wedge summand of the result. So $R\simeq \bigvee^{N} S^2$. 
The number of $2$-spheres is
$N = 2n+\sum_{i,j} C_{ij}$. Here $2n$ corresponds to the $n$ pairs of $2$-spheres $A_i, B_i$, and $C_{ij}$ is the number of circles formed by the boundaries of the finger and Whitney discs in $A_i$ and $B_j$, equal to $0$ or $1$ for each pair $i,j$ in the present context. Since Quinn's arc condition holds, note that $C_{ii}=0$ for $i=1,\dots,n$. 

By construction, the homotopy classes $\theta_{i,j}$ represented by the $S^2$ summands corresponding to the intersections $A_i\cap B_j$, $i\neq j$ are trivial in $\pi_2(X)$. Therefore the image of $\pi_2(R)$ in $\pi_2(M_{1/2})$ consists of the hyperbolic pairs represented by the spheres~$A, B$.

Consider the $5$-dimensional $h$-cobordism $X\times I\times \{ 1/2\}$. The $4$-manifold $X$ at the top is obtained by attaching $5$-dimensional $3$-handles to the middle-middle level $M_{1/2}$ along the spheres $B$; the copy of $X$ at the bottom is obtained by attaching $5$-dimensional $3$-handles (upside-down $2$-handles) along the spheres $A$. We consider the sub-$h$-cobordism~$Y$ obtained by attaching these $3$-handles to~$R\times [1/2-\varepsilon, 1/2+\varepsilon ]$. 

Denote the resulting $4$-manifolds $Y \cap (X \times \{1\}$ and $Y \cap (X \times \{1\})$ at the top and at the bottom both by $\mathcal{B}$. 
Because the $2$- and $3$-handles of $Y$  geometrically cancel, $Y$ is a product cobordism and the manifolds at either end are diffeomorphic, hence it makes sense to denote both by $\mathcal{B}$. 

The sub-$h$-cobordism $Y$ is obtained by attaching $2n$ 3-handles to $R \times [1/2-\varepsilon, 1/2+\varepsilon ] \simeq \bigvee^N S^2$, with attaching maps homotopic in $\bigvee^N S^2$ to the first $2n$ wedge summands. Hence~$Y$ has homotopy type $Y \simeq \bigvee^{N-2n} S^2$. Thus since $Y$ is an $h$-cobordism, we also have \[\mathcal{B} \simeq \bigvee^{N-2n} S^2 \simeq \bigvee^{\sum_{i,j=1}^n C_{ij}} S^2.\] 
Since $C_{ii}=0$ and $C_{ij} \leq 1$ for $i \neq j$, it follows that $\sum_{i,j} C_{ij} \leq n(n-1)$, so taking $k := \sum_{i,j} C_{ij}$ we have   $\mathcal{B} \simeq \bigvee^{k} S^2$ for some $k \leq n(n-1)$, as desired. 

Since $\mathcal{B}$ is obtained from $R$ by surgering out the spheres $A, B$, the remaining copies of $S^2$ in the wedge sum are those with homotopy classes determined by $\theta_{i,j} =0 \in \pi_2(X)$. 
It follows that the inclusion $\mathcal{B}\to X$ is null-homotopic, as claimed in the theorem.

The conclusion of the proof mirrors that of \cref{thm:corks-for-PIs-body,thm:one-eye-thm} in \cref{section:1-eyed-cork-thm}. In more detail, by \cref{lem:Quinn Core Lemma} our pseudo-isotopy $F$ is isotopic to $F'$ such that $F' = \Id$ on $(X \sm \mathring{\mathcal{B}}) \times I$.
Restricting this isotopy to $X \times \{1\}$ gives an isotopy from $f$ to $f'$ such that~$f' = \id$ on $(X \sm \mathring{\mathcal{B}}) \times \{1\}$. 
\end{proof}

\section{Examples of diffeomorphisms that are 1-stably trivial}\label{section:examples-1-stable}

We give an exposition of examples of exotic diffeomorphisms of simply-connected 4-manifolds that are 1-stably isotopic to the identity. The first examples of exotic diffeomorphisms are due to Ruberman~\cite{ruberman-isotopy}, in 1998. A year later, in \cite{Ruberman-polynomial-invariant}, he produced an infinitely generated  subgroup of $\pi_0 \Diff(Z_n)$, for each $n \geq 2$ where $Z_n := \#^{2n} \CP^2 \#^{10n+1} \ol{\CP}^2$ for $n$ odd and $Z_n := \#^{2n} \CP^2 \#^{10n+2} \ol{\CP}^2$ for $n$ even. Ruberman used Donaldson invariants to prove that his diffeomorphisms are nontrivial. In 2020, using Seiberg-Witten theory Baraglia, and Konno~\cite{Baraglia-Konno}, constructed more examples of exotic diffeomorphisms on closed 4-manifolds. Later, Iida, Konno, Taniguchi, and the second-named author~\cite{iida2022diffeomorphisms} detected exotic diffeomorphisms on 4-manifolds with nonempty boundaries using Kronheimer-Mrowka's invariant.  

The material in this section is known to the experts, however it has not all appeared in writing, and we want to give a self-contained description of the examples to which one can apply \cref{thm:one-eye-thm}. 
We will adapt an argument of Auckly-Kim-Melvin-Ruberman~\cite{AKMR-stable-isotopy}*{Theorem~C} (cf.~\cite{auckly:stable-surfaces}*{p.~6}) to check that the examples are 1-stably smoothly isotopic to the identity, and hence \cref{thm:one-eye-thm} applies nontrivially to all of these exotic diffeomorphisms.

Our exposition of the construction of the diffeomorphisms will be similar to that in Baraglia-Konno~\cite{Baraglia-Konno}, but we need a description of the `reflection' maps in terms of surfaces, in order to prove the existence of 1-stable isotopies. 

\subsection{Construction of diffeomorphisms}

Let $X_0,X_1,\dots$ be a (possibly infinite) family of closed, smooth, simply-connected 4-manifolds. Let $W$ be another closed, smooth, simply-connected 4-manifold.  Suppose that for all $p > 0$ there are orientation preserving diffeomorphisms
\[\varphi_p \colon X_p \# W \xrightarrow{\cong} X_0 \# W.\]
We suppose also that there are smoothly embedded 2-spheres $\xi_+$ and $\xi_-$ in $W$, with normal Euler number either $\pm 1$ or $\pm 2$. We can consider these spheres in $X_p \# W$, and then we denote them as $\xi^p_{\pm}$, for any $p \geq 0$. 

For each $p>0$ we require that
\[[\varphi_p(\xi_{\pm}^p)] = [\xi^0_{\pm}] \in H_2(X_0\# W;\Z).\]
In practice this can usually be arranged using Wall's results in \cite{Wall-diffeomorphisms-4-manifolds}, because the explicit $X_0 \# W$ we use will satisfy the hypotheses of Wall's theorem, and in particular will have an $S^2 \times S^2$ connected summand. 

Given a smoothly embedded $\pm 1$- or $\pm 2$-sphere $\zeta$ in a 4-manifold $M$, there is a diffeomorphism 
\[R^M_\zeta \colon M \to M\]
whose definition we explain now.  The integer $\pm 1$- or $\pm 2$ is the Euler number of the normal bundle of the sphere.  
In all cases, we will define an orientation-preserving diffeomorphism of a closed regular neighborhood $\ol{\nu}\zeta$. The boundary $\partial\ol{\nu}\zeta$ is either $S^3$ or $\RP^3$, for normal Euler number $\pm1$ or $\pm 2$ respectively.

In both cases, the diffeomorphism of $\ol{\nu}\zeta$ will restrict to a diffeomorphism isotopic to the identity on the boundary by Cerf~\cite{Cerf1959-3sphere} and Bonahon~\cite{bonahon} respectively.
%every orientation preserving diffeomorphism of these 3-manifolds is isotopic to the identity. 
Hence, after an isotopy  in a collar neighborhood of the boundary the diffeomorphism 
may be assumed to be the identity in a neighborhood of the boundary of $\ol{\nu}\zeta$. The choice of isotopy to achieve this is not unique, because $\pi_1 \Diff^+(S^3) \cong \Z/2$ and $\pi_1\Diff^+(\RP^3) \cong \Z/2 \times \Z/2$, by ~\cites{HKMR-diffeos-elliptic-3-mflds,bamler2019ricci, smaleconj}.
In the case of a $\pm 1$ sphere, this choice will not affect the isotopy class of the resulting diffeomorphism of $\ol{\nu}\zeta$.  In the case of a $\pm 2$ sphere we will carefully choose an isotopy to uniquely specify a diffeomorphism. 
To define $R^M_\zeta$, we then extend by the identity on $M \sm \nu \zeta$. 

\begin{enumerate}
    \item Suppose the normal  Euler number of $\zeta$ is $\pm 1$. Then $\ol{\nu} \zeta$ is diffeomorphic to a punctured $\CP^2$ or $\ol{\CP}^2$, via a diffeomorphism identifying $\zeta$ with $\CP^1$.  
    We focus on the Euler number $-1$ case. Complex conjugation on all homogeneous coordinates defines a diffeomorphism $\ol{\CP}^2 \to \ol{\CP}^2$, that restricts to the antipodal map on $\CP^1$. Isotope this to fix a ball, remove the ball, and we obtain a diffeomorphism of $\ol{\CP}^2 \sm \mathring{D}^4$. This determines the desired diffeomorphism of $\ol{\nu}\zeta$, uniquely up to isotopy.  A priori the choice of isotopy to fix a ball matters, but in fact this choice of isotopy is irrelevant, because there is a circle action on $\ol{\CP}^2$ that undoes a Dehn twist on the $S^3$ boundary of $\CP^2 \sm \mathring{D}^4$; see~\cite{Giansiracusa-stable-MCG}*{Theorem~2.4}, \cite{AKMR-stable-isotopy}*{Theorem~5.3}.  
    
    \item Suppose the normal Euler number of $\zeta$ is $\pm 2$. Then $\nu \zeta$ is diffeomorphic to $TS^2$ or to $T^*S^2$ respectively. Now we will apply a \emph{symplectic Dehn twist}~\cite{Arnold1995}. The symplectic version is defined on the cotangent bundle $T^*S^2$, but we focus on  $TS^2$,  following the smooth description by Auckly~\cite{auckly:stable-surfaces}.  Let $\alpha \colon S^2 \to S^2$ be the antipodal map, which is degree $-1$. Then $d\alpha \colon TS^2 \to TS^2$ restricts to a diffeomorphism $d\alpha| \colon DTS^2 \to DTS^2$ of the unit disc bundle.  Consider the sphere bundle $STS^2$, which can be identified with $SO(3)$ by considering $STS^2 = \{(u,v) \in \R^3 \times \R^3 \mid \|u\| = \|v\| =1, u \cdot v =0\}$, and sending $(u,v) \in STS^2$ to the orthonormal frame $(u,v,u \times v)$.  We can describe the action of the map $d\alpha|_{STS^2}$ via an action on $SO(3)$, and it acts as multiplication by the $3 \times 3$ diagonal matrix $\operatorname{Diag}(-1,-1,1)$.  Acting by 
    \[\left(\begin{smallmatrix} \cos (\pi(1+t)) & \sin (\pi(1+t)) & 0 \\ -\sin(\pi(1+t)) & \cos(\pi(1+t)) & 0 \\ 0 & 0 & 1 \end{smallmatrix}\right), \]
for $t \in [0,1]$, interpolates between acting by $\operatorname{Diag}(-1,-1,1)$ and by the $3 \times 3$ identity matrix~$I_3$. We insert this isotopy into an interior collar of $DTS^2$, to obtain a diffeomorphism $DTS^2 \to DTS^2$ that acts as the antipodal map on the zero section and is the identity on $STS^2 = \partial DTS^2$. 

A similar construction using the pullback $d^*\alpha \colon T^*S^2 \to T^*S^2$ yields the symplectic Dehn twist when the normal Euler number is $-2$.
\end{enumerate}

This completes our description of the reflection maps $R^M_\zeta \colon M \to M$. The induced maps on second homology are as follows~\cites{ruberman-isotopy,Ruberman-polynomial-invariant,auckly:stable-surfaces}. 
\begin{enumerate}
    \item For $\zeta$ a $\pm 1$ sphere, we have that 
    \begin{align*}
        (R^M_\zeta)_* \colon H_2(M;\Z) & \to H_2(M;\Z) \\
        x &\mapsto x \mp 2(x\cdot \zeta) \zeta.
    \end{align*}
     \item For $\zeta$ a $\pm 2$ sphere, we have that 
    \begin{align*}
        (R^M_\zeta)_* \colon H_2(M;\Z) & \to H_2(M;\Z) \\
        x &\mapsto x \mp (x\cdot \zeta) \zeta.
    \end{align*}
\end{enumerate}
Now we consider these maps for $M = X_p \# W$ and $\zeta = \xi_{\pm}^p$. For each $p \geq 0$ we define
\[\rho^p := R^{X_p\# W}_{\xi_{+}^p} \circ R^{X_p\# W}_{\xi_{-}^p} \colon X_p \# W \to X_p \# W.\]
Note that $\rho^p$ is supported in the $W$ summand of $X_p \# W$. 
For $p>0$ we define 
\[f_p := \varphi_p \circ \rho^p \circ \varphi_p^{-1} \circ (\rho^0)^{-1} \colon X_0 \# W \to X_0 \# W.\]
For suitable choices of $\{X_p\}_{p \geq 0}$, we will show that these provide examples of 1-stably trivial exotic diffeomorphisms, and hence \cref{thm:one-eye-thm} applies nontrivially to them.

\subsection{Topological and 1-stable isotopy}

The following computation will be useful in this section. 
\begin{align}\label{eqn:useful-computation-diffeos}
\begin{split}
    \varphi_p \circ \rho^p \circ \varphi_p^{-1} %&=
 %   \varphi_p \circ  R^{X_p\# W}_{\xi_{+}^p} \circ R^{X_p\# W}_{\xi_{-}^p}  \circ  \varphi_p^{-1} \\
    &= \varphi_p  \circ  R^{X_p\# W}_{\xi_{+}^p} \circ \varphi_p^{-1}  \circ \varphi_p \circ  R^{X_p\# W}_{\xi_{-}^p}  \circ   \varphi_p^{-1} \\ & \sim R^{X_0\# W}_{\varphi_p(\xi_{+}^p)}  \circ  R^{X_0\# W}_{\varphi_p(\xi_{-}^p)} \colon X_0 \# W \to X_0 \# W.
 \end{split}
   \end{align}
This relies on the observation that conjugating a reflection map $R^{X_p\# W}_{\xi_{\pm}^p}$ defined using the tubular neighborhood of an embedded sphere by the diffeomorphism $\varphi_p$, is the same as applying the analogous reflection map to the image of that tubular neighborhood, which by uniqueness of tubular neighborhoods is isotopic to applying the reflection map using any preferred tubular neighborhood.

\begin{lemma}\label{lemma:top-isotopic-Id}
    For each $p > 0$, $f_p \colon X_0 \# W \to X_0 \# W$ acts as the identity on $H_2(X_0\#W;\Z)$, and hence is homotopic,  pseudo-isotopic, topologically isotopic, and smoothly stably isotopic to the identity.   
\end{lemma}

\begin{proof}
    Since $[\varphi_p(\xi_{\pm}^p)] = [\xi^0_{\pm}] \in H_2(X_0\# W;\Z)$, we have that
    \[(R^{X_0\# W}_{\varphi_p(\xi_{\pm}^p)})_* = (R^{X_0\# W}_{\xi_{\pm}^0})_* \colon H_2(X_0\# W;\Z) \to H_2(X_0\# W;\Z),\]
and hence using \eqref{eqn:useful-computation-diffeos} 
\[ (\varphi_p \circ \rho^p \circ \varphi_p^{-1})_* = ( R^{X_0\# W}_{\varphi_p(\xi_{+}^p)})_*  \circ  (R^{X_0\# W}_{\varphi_p(\xi_{-}^p)})_* = (R^{X_0\# W}_{\xi_{+}^0})_* \circ (R^{X_0\# W}_{\xi_{-}^0})_* = \rho^0_*.\]
It follows that $(f_p)_* = \Id_{H_2(X_0\# W;\Z)}$.  
Thus $f_p$ is homotopic to the identity by~\cite{Cochran-Habegger}, topologically isotopic to the identity by Perron-Quinn~\cites{Perron,Quinn:isotopy}, and smoothly stably isotopic to the identity by Quinn-Gabai~\cites{Quinn:isotopy,GGHKP,gabai20223spheres}.
\end{proof}

\begin{definition}
For a closed, orientable 4-manifold $M$, we say that a homology class $\zeta \in H_2(M;\Z)$ is \emph{characteristic} if $\zeta \cdot x \equiv x \cdot x \mod{2}$ for all $x \in H_2(M;\Z)$. If $\zeta$ is not characteristic, then we say that $\zeta$ is \emph{ordinary}. 
\end{definition}

 For any diffeomorphism $g \colon M \to M$, let \[\wh{g} := g \# \Id_{S^2 \times S^2} \colon M \# (S^2 \times S^2) \to M \# (S^2 \times S^2).\] 
This notation will be useful in the proof of the next lemma, which gives hypotheses implying our diffeomorphisms are 1-stably isotopic to $\Id$. The proof is essentially the same as that in~\cite{auckly:stable-surfaces}.

\begin{lemma}\label{lemma:1-stably-isotopy}
Suppose that $[\xi^0_{\pm}] =[\varphi_p(\xi_{\pm}^p)]   \in H_2(X_0\# W;\Z)$ is ordinary, and that $\pi_1(W \sm \xi_{\pm}) = \{1\}$. 
Then for each $p > 0$, $f_p \colon X_0 \# W \to X_0 \# W$ is 1-stably isotopic to the identity. 
\end{lemma}

\begin{proof}
    By the main result of \cite{akmrs:one-is-enough}, and since $[\xi^0_{\pm}] =[\varphi_p(\xi_{\pm}^p)]$ is ordinary and the complement of these spheres is simply-connected, we have that 
    $\xi^0_{\pm}$ and $\varphi_p(\xi_{\pm}^p)$ are smoothly isotopic in $X_0 \# W \# (S^2 \times S^2)$.  (Special cases appeared earlier in \cite{AKMR-stable-isotopy} and \cite{Akbulut-isotoping-2-spheres}.)
Hence using~\eqref{eqn:useful-computation-diffeos} 
\begin{align*}
    f_p \# \Id_{S^2 \times S^2} &= \wh{f}_p  = \wh{\varphi}_p \circ \wh{\rho^p} \circ \wh{\varphi}_p^{-1} \circ (\wh{\rho^0})^{-1} \\
    &= \wh{R}^{X_0 \# W}_{\varphi_p(\xi^p_+)} \circ \wh{R}^{X_0 \# W}_{\varphi_p(\xi^p_-)} \circ (\wh{R}^{X_0 \# W}_{\xi^0_-})^{-1}  \circ (\wh{R}^{X_0 \# W}_{\xi^0_+})^{-1} \\
    & \sim \wh{R}^{X_0 \# W}_{\xi^0_+} \circ \wh{R}^{X_0 \# W}_{\xi^0_-} \circ (\wh{R}^{X_0 \# W}_{\xi^0_-})^{-1}  \circ (\wh{R}^{X_0 \# W}_{\xi^0_+})^{-1} = \wh{\Id}_{X_0 \# W} \\ &= \Id_{X_0 \# W \# (S^2 \times S^2)} \colon X_0 \# W \# (S^2 \times S^2) \to X_0 \# W \# (S^2 \times S^2).
\end{align*}
\end{proof}

\subsection{Examples}

The first examples of exotic diffeomorphisms of 4-manifolds were given in \cite{ruberman-isotopy}. This paper does not specify particular 4-manifolds, so we consider the examples in \cite{Ruberman-polynomial-invariant} instead. 

\begin{example}[Ruberman]\label{Example:ruberman}
    For some $n \geq 2$, let $X_0 := \#^{2n-1} \CP^2 \#^{10n-1} \ol{\CP}^2$ if $n$ is odd, and let $X_0 := \#^{2n-1} \CP^2 \#^{10n} \ol{\CP}^2$ if $n$ is even. Let $W := \CP^2 \# \ol{\CP}^2 \# \ol{\CP}^2$.  Let $\xi_{\pm} \subseteq W$ be the standard embedded sphere representing $(1, \pm 1,1) \in H_2(W;\Z) \cong \Z^3$, with the basis corresponding to the connected sum decomposition. We have that $\xi_{\pm} \cdot \xi_{\pm} =-1$, and that $\pi_1(W \sm \xi_{\pm}) = \{1\}$. Since $X_0$ is not spin,  it follows that $\xi_{\pm}$ is ordinary.  
    For $p \geq 1$, let $X_{p} := E(n;p+1)$, the result of a multiplicity $p$ log transform on the elliptic surface~$E(n)$.  Then for $Z_n :=  \#^{2n} \CP^2 \#^{10n+1} \ol{\CP}^2$ if $n$ is odd, and $Z_n := \#^{2n} \CP^2 \#^{10n+2} \ol{\CP}^2$ if $n$ is even, we have, as shown by Ruberman~\cite{Ruberman-polynomial-invariant}, an infinitely generated subgroup of $\pi_0 \Diff^+(Z_n)$, generated by $\{f_p\}_{p=1}^{\infty}$.  All of these diffeomorphisms are topologically and 1-stably isotopic to $\Id_{Z_n}$ by \cref{lemma:top-isotopic-Id,lemma:1-stably-isotopy}; note that this was already known and due to \cite{AKMR-stable-isotopy}.  Hence by \cref{thm:one-eye-thm} they each admit a diff-cork. The same strategy works for the examples of exotic diffeomorphisms on compact 4-manifolds with nonempty boundaries from~\cite{iida2022diffeomorphisms}. 
\end{example}

%Next we consider the examples from Baraglia-Konno ~\cite{Baraglia-Konno}.

\begin{example}[Baraglia-Konno]\label{Example:BK}
We present the Baraglia-Konno examples of exotic diffeomorphisms~\cite{Baraglia-Konno}.  For some $n \geq 2$, let $X_0 := \#^{n-1} (S^2 \times S^2) \#^n E(2)$. Let $X_1 := E(2n)$. 
Let $W:= S^2 \times S^2$, and let $\xi_{\pm} \subseteq W$ be an embedding representing $(1,\pm 1)$ in $H_2(S^2 \times S^2;\Z) \cong \Z^2$ with the standard basis. Then $\xi_{\pm} \cdot \xi_{\pm} = \pm 2$. We can choose $\xi_{\pm}$ such that $S^2 \times \{\pt\}$, which is embedded with trivial normal bundle, intersects $\xi_{\pm}$ exactly once, and hence $\xi_{\pm}$ is ordinary and has $\pi_1(W \sm \xi_{\pm})=\{1\}$. With this data, $f_1 \colon \#^{n} (S^2 \times S^2) \#^n E(2) \to \#^{n} (S^2 \times S^2) \#^n E(2)$ is a diffeomorphism that is topologically and 1-stably isotopic to the identity by \cref{lemma:top-isotopic-Id,lemma:1-stably-isotopy}. Baraglia-Konno proved that $f_1$ is not smoothly isotopic to the identity. By \cref{thm:one-eye-thm} they each admit a diff-cork. 
\end{example}

\begin{example}[Auckly]\label{Example:Auckly}
Further examples  were given by Auckly~\cite{auckly:stable-surfaces}, also making use of \cite{Baraglia-Konno}*{Theorem~4.1}, 
using $X_0 = E(2)$ and $X_p := E(2;2p+1)$ for $p \geq 1$. This yields a family of exotic and 1-stably isotopic diffeomorphisms, similarly to  \cref{Example:ruberman}, that generate an infinite rank subgroup in the abelianization of the mapping class group. By \cref{thm:one-eye-thm} they each admit a diff-cork. 
\end{example}

We note that our description of the $R_{\zeta}$ diffeomorphisms differs from that in \cite{Baraglia-Konno}. However this does not affect our ability to apply the results of \cite{Baraglia-Konno}.  
The key property of the $R_{\zeta}$ diffeomorphisms is their action on homology, which is such that the Stiefel-Whitney class $w_1(H^+)$ appearing in the  Baraglia-Konno gluing formula is nonzero~\cite{Baraglia-Konno}*{Theorem~4.1}. This nonvanishing in the Baraglia-Konno formula enables one to  express the Family Seiberg-Witten invariants of the diffeomorphisms $f_p$ in terms of the Seiberg-Witten invariants of the manifolds $\{X_p\}_{p \geq 0}$. Since the manifolds $\{X_p\}_{p \geq 0}$ have pairwise distinct $\Z/2$-Seiberg-Witten invariants, we can deduce that the diffeomorphisms in question are pairwise not smoothly isotopic.  However \cref{thm:one-eye-thm} applies to them all.

\Zmdiffcork*

\begin{proof}
Consider $f_1,\dots,f_m \colon Z_n \to Z_n$, the first $m$ of Ruberman's family of diffeomorphisms from \cref{Example:ruberman}. Let $\operatorname{TDiff}(Z_n) \subseteq \Diff^+(Z_n)$ be the Torelli subgroup of diffeomorphisms acting trivially on integral homology. Ruberman defined a group homomorphism $\mathbb{D} \colon \pi_0\operatorname{TDiff}(Z_n) \to \R[\![H_2(Z_n)^*]\!]$ valued in a power series ring, and showed that $\{\mathbb{D}(f_p)\}_{p=1}^m$ is a linearly independent set in $\R[\![H_2(Z_n)^*]\!]$. 

By~\cref{thm:one-eye-thm-full-version}, there is a compact, contractible, smooth, codimension zero submanifold $\mathcal{C}_m$ such that each $f_i$, for $i=1,\dots,m$, can be isotoped to a diffeomorphism $g_i$ that is  supported on $\mathcal{C}_m$. 
We consider~$g_i$ as a diffeomorphism of~$\mathcal{C}_m$.  Now consider the composition \[\pi_0 \Diff_{\partial}(\mathcal{C}_m) \xrightarrow{\iota} \pi_0\operatorname{TDiff}(Z_n) \xrightarrow{\mathbb{D}} \R[\![H_2(Z_n)^*]\!].\]
Since $\iota(g_i) \sim f_i$, the set $\{\mathbb{D} \circ \iota(g_i)\}_{i=1}^m$ is linearly independent. 
It follows that $\{g_i\}_{i=1}^m$ generates a subgroup of  $\pi_0 \Diff_{\partial}(\mathcal{C}_m)$ that abelianizes to $\Z^m$. 
\end{proof}

\begin{remark}
    One can also use the diffeomorphisms from \cref{Example:Auckly} to prove \cref{thm:Zm-in-a-diff-cork}, by combining with \cite{konno-lin2023hom-instab}*{Proposition~5.4}. 
\end{remark}

\section{Monopole Floer homology, family Seiberg-Witten invariants, and applications}\label{section:FSW-applications}

In this section we recall the monopole Floer cobordism maps, the family Seiberg-Witten invariant, and a gluing formula of Lin~\cite{lin2022family}. 
We then prove \cref{thm:inclusion-not-he}, and we show that the family cobordism maps associated with the diff-corks arising from  \cref{Example:ruberman,Example:BK,Example:Auckly} are nontrivial.  

\subsection{Monopole Floer cobordism maps}\label{section:monopole-floer}

Our exposition here follows that of Jianfeng Lin in \cite{lin2022family}.

First, recall that given a closed, oriented 3-manifold $Y$ with a $\spin^c$ structure, Kronheimer and Mrowka~\cite{monopole-KM} defined abelian groups called the \emph{monopole Floer (co)homology} of $(Y,\mathfrak{s})$. These groups come in various flavors. 
We will use the homology $\widehat{HM}_*(Y,\mathfrak{s})$ and $\HMt_*(Y,\mathfrak{s})$ and the corresponding cohomology groups $\widehat{HM}^*(Y,\mathfrak{s})$ and  
$\HMt^*(Y,\mathfrak{s})$. 
%$\widebreve{HM}^*(Y,\mathfrak{s})$. 
Following Lin, since we need to use his gluing formula, we will work over $\Q$. So these groups are $\Q$-vector spaces. 

From now on assume that $Y$ is an integral homology 3-sphere. Then these groups are $\Z$-graded $\Q$-vector spaces \cite{monopole-KM}*{28.3.3}, 
and moreover there is a graded module structure over the polynomial ring $\Q[U]$. For homology the action of $U$ has degree $-2$, while for cohomology the action of $U$ has degree~$2$.  

An integral homology sphere $Y$ admits a unique $\spin^c$ structure, and so in such cases we will omit $\mathfrak{s}$ from the notation. With this in mind, we have an isomorphism of graded $\Q[U]$-modules
$\widehat{HM}_*(S^3) \cong \Q[U]\langle -1 \rangle$, where $\langle -1 \rangle$ denotes a grading shift, and implies that the constants live in degree~$-1$. Since $U$ has degree $-2$, it follows that $\widehat{HM}_*(S^3)$ consists of a copy of $\Q$ in each odd negative degree. We let $\widehat{1}$ denote the canonical generator in degree~$-1$.  
%MP: According to Lin's notes, we should think of $\Q[U]$ as shifted down by one, so that the constant term is in degree $-1$
%
On the other hand, we have an isomorphism of graded $\Q[U]$-modules 
$\HMt_*(S^3) \cong \Q[U,U^{-1}]/ U\cdot\Q[U]$,
so there is a copy of $\Q$ in each nonnegative even degree.

The cohomology $\HMt^*(S^3)$ is isomorphic to $\HMt_*(S^3)$ as a graded vector space, but now the action of $U$ has degree $2$, and so instead we have an isomorphism of graded $\Q[U]$-modules $\HMt^*(S^3) \cong \Q[U]$. We let $\check{1}$ denote the canonical generator of $\HMt^*(S^3)$, which lies in degree $0$.
Similarly, the cohomology $\widehat{HM}^*(S^3)$ is isomorphic to $\Q[U,U^{-1}]\langle -1 \rangle/U \cdot \Q[U]$, with a copy of $\Q$ in each negative odd degree.

\begin{remark}\label{remark:degree-nonnegative-implies-zero}
  Every graded homomorphism $\HMt^*(S^3) \to \widehat{HM}^*(S^3)$ of nonnegative degree is trivial, because every nonzero element of $\HMt^*(S^3)$ lies in positive grading, hence the image in $\wh{HM}^*(S^3)$ is positively graded. But $\wh{HM}^*(S^3)$ only has nontrivial groups in negative degrees. 
\end{remark}

Now let $M$ be a smooth, closed, oriented 4-manifold with $b_2^+(M) > 2$, and let $g \in \Diff(M)$ be a diffeomorphism that induces the identity map on homology $H_*(M)$, and that fixes a 4-ball $\Delta \subseteq M$ pointwise.   
Consider the $M$-bundle $p \colon \widetilde{M} \to S^1$ obtained by taking the mapping torus of $g$. 
Assume that the bundle $\wt{M}$ admits a decomposition \[\wt{M} = \wt{M}_0 \cup_{Y \times S^1} M_1 \times S^1\]
into an $M_0$-bundle over $S^1$ and a trivial $M_1$-bundle over $S^1$, where $\partial M_0 = Y$ is an oriented integral homology 3-sphere, $b_2^+(M_1) > 1$, and $\partial M_1 =-Y$.  Assume additionally that $\Delta \subseteq M_0$, so $\wt{M}_0$ has a trivial sub-bundle $\Delta \times S^1$.  Removing $\mathring{\Delta} \times S^1$ from  $\wt{M}_0$ yields a bundle of cobordisms~$\wt{W}_0$ from $S^3$ to $Y$, over~$S^1$. Removing a 4-ball from $M_1$ yields a cobordism $W_1$ from $Y$ to $S^3$. 

We consider a family $\spin^c$ structure $\wt{\mathfrak{s}}$ on $\wt{M} \to S^1$, and restrict it to a family $\spin^c$ structure $\wt{\mathfrak{s}}_0$ on $\wt{W}_0$ and to a $\spin^c$ structure on $\mathfrak{s}_1$ on $W_1$.  
 There is an induced family cobordism map
\[\widehat{HM}_*(\wt{W}_0,\widetilde{\mathfrak{s}}_0) \colon \widehat{HM}_*(S^3)\to \widehat{HM}_*(Y)\]
%\[FSW(g,\widetilde{\mathfrak{s}}_0) \colon \widehat{HM}_*(S^3)\to \widehat{HM}_*(Y)\]
on monopole Floer homology \cite{lin2022family}*{Proposition 4.5}, and there is an induced cobordism map 
\[\vv{HM}^*(W_1,\mathfrak{s}_1) \colon \HMt^*(S^3)\to \widehat{HM}^*(Y)\]
on monopole Floer cohomology \cite{monopole-KM}*{Section~3.5}.  
For the proof of \cref{thm:inclusion-not-he}, we need to investigate the effect of the latter map on gradings.

\begin{remark}\label{remark:dependence-on-choice-of-isotopy}
Note that given a diffeomorphism of $g' \colon M \to M$, there is a choice of an isotopy of $g'$  to a diffeomorphism $g$ fixing a 4-ball $\Delta \subseteq M$. The family cobordism map $\widehat{HM}_*(\wt{W}_0,\widetilde{\mathfrak{s}}_0)$ \emph{a priori} depends on these choices. While the analysis of this dependence is outside the scope of this paper, the results below hold for any choice. 
\end{remark}

\begin{lemma}\label{lemma:degree-of-map}
   The map $\vv{HM}^*(W_1,\mathfrak{s}_1)$ is a graded homomorphism  of degree $-d(W_1,\mathfrak{s}_1)$, where 
    \[d(W_1,\mathfrak{s}_1) := \frac{c_1(\mathfrak{s})^2[W_1] -2\chi(W_1) -3\sigma(W_1)}{4}.\]
\end{lemma}

Here $c_1(\mathfrak{s}_1)$ is the first Chern class, $\chi(W_1)$ is the Euler characteristic, and $\sigma(W_1)$ is the signature of the intersection pairing.

\begin{proof}
This fact is contained in \cite{monopole-KM}; for non-experts we explain how to extract it.  
The degree of the map on cohomology is the negative of the map on homology, so it suffices to see that $\vv{HM}_*(W_1,\mathfrak{s}_1) \colon \widehat{HM}_*(Y) \to \HMt_*(S^3)$ has degree $d(W_1,\mathfrak{s}_1)$.   By \cite{monopole-KM}*{Theorem~3.5.3}, there is a map $j_* \colon \HMt_*(Y) \to \widehat{HM}_*(Y)$ of degree 0~\cite{monopole-KM}*{p.~52} such that 
\[\vv{HM}_*(W_1,\mathfrak{s}_1) \circ j_* = \HMt_*(W_1,\mathfrak{s}_1) \colon \HMt_*(Y) \to \HMt_*(S^3).\]
So it suffices to see that $\HMt_*(W_1,\mathfrak{s}_1)$ has degree $d(W_1,\mathfrak{s}_1)$.  

By \cite{monopole-KM}*{Equation~(28.3), p.~588}, the degree of $\HMt_*(W_1,\mathfrak{s}_1)$  is $d:= \smfrac{1}{4}c_1(\mathfrak{s}_1)^2[W_1] - \iota(W_1) - \smfrac{1}{4}\sigma(W_1)$, where $\iota(W_1) := \smfrac{1}{2}(\chi(W_1) + \sigma(W_1) + \beta_1(S^3) - \beta_1(Y))$ (see \cite{monopole-KM}*{Definition~25.4.1}).  A straightforward calculation shows that $d = d(W_1,\mathfrak{s})$, as required.  
\end{proof}

\begin{corollary}\label{cor:vanishing-cobordism-map}
    If $Y=S^3$ and $d(W_1,\mathfrak{s}_1) \leq 0$, then $\vv{HM}^*(W_1,\mathfrak{s}_1) \colon \HMt^*(S^3) \to \HMf^*(S^3)$ is the zero map.
\end{corollary}

\begin{proof}
    If $d(W_1,\mathfrak{s}_1) \leq 0$ then $-d(W_1,\mathfrak{s}_1) \geq 0$, so by \cref{lemma:degree-of-map} the degree of $\vv{HM}^*(W_1,\mathfrak{s}_1)$ is nonnegative. The corollary then follows from \cref{remark:degree-nonnegative-implies-zero}. 
\end{proof}

\subsection{The family Seiberg-Witten invariant}\label{section:family-SW-intro}

We continue with the notation and assumptions from the previous subsection.
In addition, assume that the family expected dimension of $(\wt{M},\wt{\mathfrak{s}})$ vanishes, that is:
\begin{equation}\label{eqn:expected-dimension}
    d(\wt{M},\wt{\mathfrak{s}}) := \frac{c_1(\wt{\mathfrak{s}}|_M)^2[M] -2\chi(M) -3 \sigma(M)}{4} + 1 = 0. 
\end{equation} 
In this context, in particular given $M$ with $b_2^+(M)>2$, a diffeomorphism $g \colon M \to M$ that acts trivially on homology, and a family $\spin^c$ structure $\wt{\mathfrak{s}}$ with $d(\wt{M},\wt{\mathfrak{s}})=0$, one can define the \emph{family Seiberg-Witten invariant}, \cites{ruberman-isotopy,Li-Liu}, \cite{lin2022family}*{Section~2}, 
\[\FSW(g,\wt{\mathfrak{s}}) \in \Z. \]
If $g$ is smoothly isotopic to the identity in $\operatorname{Diff}(M)$, then $\FSW(g,\widetilde{\mathfrak{s}})= 0$ for any family $\spin^c$ structure $\widetilde{\mathfrak{s}}$ with $d(\widetilde{M},\widetilde{\mathfrak{s}})=0$ by~\cite{ruberman-isotopy}*{Lemma~2.7}. 
(If $d(\widetilde{M},\widetilde{\mathfrak{s}}) \neq 0$, then this also holds, by definition.) 
Our applications of family Seiberg-Witten theory are based on the following gluing formula of Jianfeng Lin. Let  $\langle -, - \rangle \colon \widehat{HM}_*(Y)  \times \widehat{HM}^*(Y) \to \Q$ denote the Kronecker pairing.  After an isotopy we assume that $g$ fixes a 4-ball $\Delta$, giving rise to a family cobordism map $\widehat{HM}_*(\wt{W}_0, \widetilde{\mathfrak{s}}_0)$; see \cref{remark:dependence-on-choice-of-isotopy}. 

\begin{theorem}[{\cite{lin2022family}*{Theorem~L}}]\label{thm:lin-gluing}
Considering $\FSW(g,\wt{\mathfrak{s}})$ as a rational number, we have 
    \[\FSW(g,\wt{\mathfrak{s}}) = \langle 
    \widehat{HM}_*(\wt{W}_0, \widetilde{\mathfrak{s}}_0)(\widehat{1}), \vv{HM}^*(W_1,\mathfrak{s}_1)(\check{1})   \rangle \in \Q.\]
\end{theorem}

\subsection{Proof of \texorpdfstring{\cref{thm:inclusion-not-he}}{Theorem 1.7}}\label{subsection:proof-inclusion-not-weak}

  Consider a compact, contractible $n$-manifold $C$.  Fix an embedding $D^n \hookrightarrow \mathring{C}$, and let $E_n \colon \Diff_\partial(D^n)\to \Diff_\partial(C)$ be the map given by extending diffeomorphisms of $D^n$ by the identity over $C \sm D^n$. 
   Galatius and Randal-Williams  \cite{galatius2023alexander}*{Theorem B} showed that $E_n$ is a weak equivalence for $n \geq 6$. 
 Our next result, whose statement we recall for the convenience of the reader, shows that $E_4$ is not a weak equivalence for suitable choices of $C$. 

\inclusionnotweak*

   \begin{proof}
   Let $f \colon X \to X$ be a diffeomorphism of a closed, simply connected 4-manifold that is 1-stably isotopic to $\Id_X$, together with a family $\spin^c$ structure $\widetilde{\mathfrak{s}}$ with $d(\widetilde{X},\widetilde{\mathfrak{s}})=0$, where $\wt{X}$ is the mapping torus of $f$, and $\FSW(f,\wt{\mathfrak{s}}) \neq 0$.  Baraglia-Konno~\cite{Baraglia-Konno} and Auckly~\cite{auckly:stable-surfaces} proved that all of the examples from \cref{Example:ruberman,Example:BK,Example:Auckly} satisfy these conditions. 

   By \cref{thm:one-eye-thm}, $f$ is isotopic to a diffeomorphism supported on a contractible codimension zero submanifold $C$. Let $h \colon C \to C$ be the restriction.  Suppose for a contradiction that there is an embedding $D^4 \subseteq C$ such that $h$ is isotopic to a diffeomorphism supported on $D^4$. Then $f$ is isotopic to a diffeomorphism  $f' \colon X \to X$ such that $f'$ is supported on $D^4$. Let $g := f'| \colon D^4 \to D^4$.   
 
Now, consider the decomposition  $X = D^4\cup_{S^3} X'$, where $X' := X \sm \mathring{D}^4$.  Let $W_0$ denote $D^4$ with a further $\mathring{D}^4$ removed from the interior, the closure of which we may assume is fixed by $g$. 
Let $W_1$ denote $X \sm (\mathring{D}^4 \sqcup \mathring{D}^4)$, namely $X'$ with a further puncture.  
Let $\wt{\mathfrak{s}}_0$ be the restriction of $\wt{\mathfrak{s}}$ to $\wt{W}_0$ and let $\mathfrak{s}_1$ be the restriction to $W_1$. 
%$\mathring{X} \subseteq X \subseteq \wt{X}$.  
Then the gluing formula for the family Seiberg-Witten invariant (\cref{thm:lin-gluing}) implies that 
 \[\FSW(f',\wt{\mathfrak{s}}) = \langle 
    \widehat{HM}_*(\wt{W}_0, \widetilde{\mathfrak{s}}_0)(\widehat{1}), \vv{HM}^*(W_1,\mathfrak{s}_1)(\check{1})   \rangle \in \Q.\]
Let  $\mathfrak{s}$ be the restriction of $\wt{\mathfrak{s}}$ to $X$, and note that in our case $b_2^+(X) > 2$ , so in particular $b_2^+(W_1) > 1$.   
Since $d(\widetilde{X},\widetilde{\mathfrak{s}})=0$, it follows from 
\eqref{eqn:expected-dimension} that the ordinary expected dimension
\[d(X,\mathfrak{s}) = \frac{c_1(\mathfrak{s})^2[X] -2\chi(X) -3 \sigma(X)}{4} = -1. \]
Replacing $X$ with $W_1$, we have that $c_1(\mathfrak{s})^2[X] = c_1(\mathfrak{s_1})^2[W_1]$ and $\sigma(X) = \sigma(W_1)$, but $\chi(W_1) = \chi(X) -2$. Hence $d(W_1,\mathfrak{s}_1) = d(X,\mathfrak{s})+1 =0$.   
By \cref{cor:vanishing-cobordism-map}, $\vv{HM}^*(W_1,\mathfrak{s}_1)$ is the zero map.
Hence by the gluing formula, $FSW(f', \widetilde{\mathfrak{s}})= 0$. Isotopy invariance of $\FSW$ and the fact that $\FSW(f,\wt{\mathfrak{s}}) \neq 0$ yields the desired contradiction. 
\end{proof}

\begin{remark}
It is a fact known to experts in gauge theory that an exotic diffeomorphism detected by 1-parameter family Seiberg-Witten invariants cannot be isotopic to one supported in a 4-ball.
    We thank David Auckly and Danny Ruberman for mentioning this fact to us, and particularly Hokuto Konno for a detailed discussion of a proof.
Since a proof has not yet appeared in the literature, we used a different, more specific argument in the previous proof that suffices for our purposes.  
    See also~\cite{linm2021family}*{Theorem~1.6} for a related statement on the Bauer-Furuta invariant.  
\end{remark}

\subsection{A diff-cork with nontrivial family monopole Floer cobordism map}\label{subsection:nontrivial-FSW}

In this section we prove the following result on the nonvanishing of a closely related monopole Floer family cobordism map. 

\begin{theorem}\label{Thm:FSW-body}
    There exists a compact, contractible, smooth 4-manifold $C$ and a diffeomorphism $g' \colon C \to C$ with $g'|_{\partial C}=\Id$, such that for any choice of isotopy of $g'$ to a diffeomorphism $g$ fixing a 4-ball $($see \cref{remark:dependence-on-choice-of-isotopy}$)$,   
    the family  cobordism map
     \[\widehat{HM}_*(\wt{W}_0, \widetilde{\mathfrak{s}_0}) \colon \wh{HM}_*(S^3) \to \wh{HM}_*(\partial C)\]
is nontrivial. 
    %such that $FSW(\tilde{f},\wt{\mathfrak{s}}_0) \neq 0$, 
    Here $W_0 := C \sm \mathring{D}^4$, $\wt{W}_0$ is the mapping torus of $g|_{W_0}$, and $\wt{\mathfrak{s}}_0$ is the family $\spin^c$ structure on $\wt{W}_0$ coming from restricting the unique  $\spin^c$ structure on $C$ and gluing using $g$.  
    %$\FSW(\widetilde{f})\neq 0$.
\end{theorem}

\begin{proof}
    In \cref{Example:ruberman,Example:BK,Example:Auckly}, we observed that there exists, due to Ruberman~\cite{Ruberman-polynomial-invariant}, Baraglia-Konno~\cite{Baraglia-Konno}, Auckly~\cite{auckly:stable-surfaces}, and \cite{AKMR-stable-isotopy}*{Theorem~C}, a smooth, closed, simply-connected 4-manifold $X$ and a diffeomorphism $f \colon X \to X$ that becomes smoothly isotopic to the identity after a single stabilization with $S^2 \times S^2$. In fact there are many possible choices for $X$ and $f$.
    
    Thus by \cref{thm:one-eye-thm}, there exists a contractible codimension zero submanifold $C \subseteq X$ such that $f$ is smoothly isotopic to a diffeomorphism $f'$ supported on $C$. Let $g:= f'|_C \colon C\to C$. Baraglia-Konno~\cite{Baraglia-Konno}*{Theorem 9.7} proved that there exists a family $\spin^c$ structure~$\widetilde{\mathfrak{s}}$ on the mapping torus $\widetilde{X}$ of $f$ such that the virtual dimension $d(\widetilde{X},\widetilde{\mathfrak{s}})=0$ and such that~$\FSW(f, \widetilde{\mathfrak{s}}) \neq 0$, and hence $\FSW(f', \widetilde{\mathfrak{s}}) \neq 0$.  
    %This implies that $f$ is not smoothly isotopic to the identity. 

 Let $\wt{W}_0$ and $\wt{\mathfrak{s}}_0$ be as in the statement of the theorem. Let $W_1 := X \sm (\mathring{C} \sqcup \mathring{D}^4)$ and let~$\mathfrak{s}_1$ denote the restriction of $\wt{\mathfrak{s}}$ to $W_1$.  Note that for these examples $b_2^+(X)>2$ and~$b_2^+(W_1) >1$.  
    \cref{thm:lin-gluing} implies that
     \[\FSW(f', \widetilde{\mathfrak{s}})= \langle 
    \widehat{HM}_*(\wt{W}_0, \widetilde{\mathfrak{s}}_0)(\widehat{1}), \vv{HM}^*(W_1,\mathfrak{s}_1)(\check{1})   \rangle.\]
    Since  $\FSW(f', \widetilde{\mathfrak{s}})\neq 0$, it follows  that the family cobordism map $\widehat{HM}_*(\wt{W}_0, \widetilde{\mathfrak{s}}_0)$ is nontrivial, as desired.
    \end{proof}

\begin{remark}
 One might attempt to cap off the examples from this theorem to create new examples of closed 4-manifolds with exotic diffeomorphisms, using \cref{thm:lin-gluing}. However, this capping-off process often leads to vanishing family Seiberg-Witten invariant. It would be intriguing to be able to realise this, or to make analogous arguments  using family Bauer-Furuta theory \cite{linm2021family}*{Theorem 1.8}, but we have not been able to achieve either.
\end{remark}

\section{Barbell diffeomorphisms}\label{section:generalised-barbell}

The aim of this section is to prove the following result, which contrasts with \cref{thm:inclusion-not-he}, since it gives an example of a 4-manifold where every exotic diffeomorphism (if any exist, which we do not know) can be isotoped to one supported in a 4-ball. We thank David Gabai for suggesting to us to try to prove this statement.

\barbellgenerates*

Before proving the theorem, we describe  barbell diffeomorphisms, and investigate their Poincar\'{e} variations. 
Let $x_1,\dots,x_n \in D^3$ be disjoint points in the interior and let $B_n$ be $D^3$ with $n$ disjoint open 3-balls, with centres the $x_i$, removed.  
After smoothing corners, $B_n \times I \cong \natural^n S^2\times D^2= X_n$.
We can think of $B_n \times I$ as obtained by removing disjoint open tubular neighborhoods of the  arcs $\gamma_i := x_i\times I \subseteq D^3 \times I$, for $i=1,\dots,n$.  We also note that $\partial X_n \cong \#^n S^2 \times S^1$. 
Consider $n$ pairwise disjoint arcs $d_1,\dots,d_n$ in $D^3$, where $d_i$ connects $x_i$ to $\partial D^3$.  Consider the discs $D_i := (d_i \cap B_n) \times I \subseteq B_n \times I \cong X_n$. The relative homology is $H_2(X_n, \partial X_n) \cong \Z^n$ generated by the classes $[D_i]$. Also $H_2(X_n) \cong \Z^n$, generated by the linking spheres $[S_i]$ to the arcs $\gamma_i$, i.e.\ $S_i$ is the boundary of the $i$th 3-ball removed from $D^3$ when forming $B_n$.

For $1 \leq i \neq j \leq n$, we now recall the barbell diffeomorphisms $\phi_{i,j} \colon X_n \to X_n$, defined by Budney and Gabai in \cite{Budney-Gabai}. Fill in the neighborhood of the $i$th arc $\gamma_i$ to obtain $B_{n-1} \times I$.   
Consider the loop  $\Upsilon_{i,j} \in \pi_1 (\Emb_{\partial}(I, B_{n-1} \times I),\gamma_i)$ obtained by taking $\gamma_i$ and lassooing the $j$th linking sphere $S_j$. The connecting homomorphism in the long exact sequence in homotopy groups of the fibration  
\[\Diff_{\partial}(B_{n} \times I) \to \Diff_{\partial}(B_{n-1} \times I) \xrightarrow{-\circ \gamma_i} \Emb_{\partial}(I, B_{n-1} \times I)\]
is a homomorphism $\delta \colon \pi_1 (\Emb_{\partial}(I, B_{n-1} \times I),\gamma_i) \to \pi_0 (\Diff_{\partial}(B_{n} \times I))$. 
It can be defined directly using the parametrized isotopy extension theorem.
We define 
\[\phi_{i,j} := \delta(\Upsilon_{i,j}) \colon X_n \to X_n.\] 
In $H_2(X_n)$ we have 
\[[D_i - \phi_{i,j}(D_i)] = [S_j],\,\, [D_j - \phi_{i,j}(D_j)] = -[S_i], \text{ and }
[D_k - \phi_{i,j}(D_k)] = 0\] 
for $k \geq i,j$.  
This determines the Poincar\'{e} variation in $\Hom(H_2(X_n,\partial X_n),H_2(X_n))$ associated with~$\phi_{i,j}$.  

For readers not familiar with it, let us recall some of the theory of Poincar\'{e} variations. 
Given a boundary-fixing homeomorphism $f \colon X \to X$ of a compact 4-manifold $X$, there is a \emph{Poincar\'{e} variation} \cite{Saeki}, which is represented by an element of $\Hom(H_2(X,\partial),H_2(X))$ given by $[y] \mapsto [y-f(y)]$. Saeki~\cite{Saeki} defined a group structure on a specified subset of $\Hom(H_2(X,\partial),H_2(X))$, giving the group of  Poincar\'{e} variations of $X$. We will not recall it here; see also \cite{Orson-Powell}.   If $f$ acts trivially on $H_2(X)$, then the group of Poincar\'{e} variations is isomorphic to  $\wedge^2 H_1(\partial X)^*$ i.e.\ the group of skew-symmetric forms $\kappa \colon H_1(\partial X) \times H_1(\partial X) \to \Z$.  
Let $\kappa^\dag$ denote the adjoint of $\kappa$. Then the Poincar\'{e} variation associated with $\kappa$ is the composite
\begin{equation}\label{eq:variation-composite}
       H_2(X,\partial X) \to H_1(\partial X) \xrightarrow{\kappa^\dag} H_1(\partial X)^* \xrightarrow{\operatorname{ev}^{-1}} H^1(\partial X) \xrightarrow{PD} H_2(\partial X) \to H_2(X).
\end{equation}
By \cite{Orson-Powell}, for compact simply-connected 4-manifolds with connected boundary, the topological boundary-fixing mapping class group $\pi_0 \Homeo_{\partial}(X)$ is isomorphic to the group of Poincar\'e variations.

\begin{proof}[Proof of \cref{thm:barbell-generate}]
In the case of $X_n$, since $H_2(\partial X_n) \to H_2(X_n)$ is onto, every boundary-fixing homeomorphism of $X_n$ acts trivially on $H_2(X_n)$, and hence the group of Poincar\'e variations, and as a consequence $\pi_0 \Homeo_{\partial}(X_n)$, is isomorphic to $\wedge^2 H_1(\partial X_n)^*$.  
One can check using \eqref{eq:variation-composite} that the variation of the barbell diffeomorphism $\phi_{i,j}$ is described by \[e_i \wedge e_j  \in \wedge^2 H^1(\partial X_n)^* \cong \wedge^2 \Z^n,\]
where $e_1,\dots,e_n$ are the standard generators of $H_1(\partial X_n)^* \cong \Z^n$.
It follows that the topological mapping class group $\pi_0 \Homeo_{\partial}(X_n) \cong \wedge^2 \Z^n$ is generated by the barbell diffeomorphisms  $\{\phi_{i,j}\}$, proving the last statement of the theorem.

Now let $f\in \ker (\pi_0\Diff_{\partial}(X_n) \to \pi_0 \Homeo_{\partial}(X_n))$.  
It follows that $f$ has trivial Poincar\'e variation. 
Using the Hurewicz theorem, this implies that $D_i$ is homotopic rel.\ boundary to $f(D_i)$ for $i=1,\dots,n$.   

Note that  $D_1$ and $f(D_1)$ have a common dual sphere in the boundary, $S^2 \times \{x\}$, for some  $x\in \partial D^2$. By the light bulb theorem~\cites{gabai:lightbulb,Kosanovic-Teichner},   
%(the quickest method is to use the light bulb theorem for discs in the latter reference, although in this case one can also apply Gabai's original version) 
we obtain a smooth isotopy between $f(D_1)$ and $D_1$. By the isotopy extension theorem and the fact that boundary-fixing diffeomorphisms of the 2-disc are all isotopic rel.\ boundary to one another, we can isotope $f$ so that it fixes $D_1$ pointwise. By uniqueness of tubular neighborhoods, we can assume that $f$ fixes a tubular neighborhood of $D_1$ setwise.  For a disc in a 4-manifold with a framing of the normal bundle restricted to its boundary, if the framing extends to the entire disc then it does so essentially uniquely, because $\pi_2(O(2))=0$. Thus we can assume after a further isotopy that $f$ fixes a tubular neighborhood of $D_1$ pointwise.  

Cutting along this tubular neighborhood of $D_1$ leaves a diffeomorphism of $X_{n-1}$ restricting to the identity on the boundary. By induction, and writing $X_0 \cong D^4$, we obtain a diffeomorphism of $X_n$ that is supported on a 4-ball $D^4$, as desired.  
\end{proof}

\def\MR#1{}
\bibliography{bib}

\end{document}